\documentclass[a4paper, 12pt]{amsart}
\newcommand{\newdefinition}{\theoremstyle{definition}\newtheorem}
\newcommand{\newproof}[2]{}

\usepackage[utf8x]{inputenc}
\usepackage{amssymb,microtype}
\usepackage{amsthm}
\usepackage{thmtools, thm-restate}

\usepackage{pifont,MnSymbol}
\usepackage{tikz}
\usepackage{sseq}
\usepackage{microtype}
\usepackage{tikz-cd}
\usepackage[english, status=draft]{fixme}
\usepackage[toc,title]{appendix}
\usepackage[colorlinks=true]{hyperref} 
\usepackage{amscd}
\fxusetheme{color}
\usepackage[colorlinks=true]{hyperref} 

\usetikzlibrary{matrix,arrows,decorations}

\DeclareMathOperator{\Ab}{Ab}

\DeclareMathOperator{\Gal}{Gal}

\DeclareMathOperator{\Ext}{Ext}
\DeclareMathOperator{\Pic}{Pic}
\DeclareMathOperator{\Ch}{Ch}
\DeclareMathOperator{\PSL}{PSL}
\DeclareMathOperator{\Drv}{D}
\DeclareMathOperator{\PGL}{PGL}

\DeclareMathOperator{\SL}{SL}
\DeclareMathOperator{\Aut}{Aut}
\DeclareMathOperator{\divis}{div}
\DeclareMathOperator{\Div}{Div}
\DeclareMathOperator{\RHom}{RHom}
\DeclareMathOperator{\Out}{Out}
\DeclareMathOperator{\Cone}{C}

\DeclareMathOperator{\Sh}{Sh}

\usepackage{amsmath,calligra,mathrsfs}
\DeclareMathOperator{\HOM}{\mathscr{H}\text{\kern -3pt {\calligra\large om}}\,}
\DeclareMathOperator{\DIV}{\mathscr{D}\text{\kern -3pt {\calligra\large iv}}\,}

\newcommand{\Hom}{{\operatorname{Hom}}}
\numberwithin{equation}{section}
\newtheorem{lemma}[equation]{Lemma}
\newtheorem{corollary} [equation]{Corollary}
\newtheorem{prop} [equation]{Proposition}

\newdefinition{example}[equation]{Example}
\newdefinition{defn}[equation]{Definition} 
\newdefinition{remark}[equation]{Remark}
\newdefinition{notation}[equation]{Notation}
\newproof{proof}{Proof}

\declaretheorem[numberwithin=section]{theorem}

\theoremstyle{definition}

\newtheorem{definition}[theorem]{Definition}

\newcommand{\ZZ}{{\mathbb Z}}

\makeatletter

\newcommand{\Spec}{\text{Spec }}
\newcommand{\et}{\'et\@ifstar{\'e}{e\xspace}}

\begin{document}


%






\author{Magnus Carlson and Tomer Schlank}
\address{Department of Mathematics, KTH Royal Institute of Technology, S-100 44 Stockholm, Sweden}
\email{macarlso@math.kth.se}
\address{Einstein Institute of Mathematics, The Hebrew University of Jeru\-sa\-lem, Jeru\-sa\-lem, 91904, Israel}
\email{tomer.schlank@mail.huji.ac.il}



\title[]{The unramified inverse Galois problem and cohomology rings of totally imaginary number fields}
\begin{abstract}
We employ methods from homotopy theory to define new obstructions to solutions of embedding problems. By using these novel obstructions we study embedding problems with non-solvable kernel. We apply these obstructions to study the unramified inverse Galois problem. That is, we show that our methods can be used to determine that certain groups cannot be realized as the Galois groups of unramified extensions of certain number fields. To demonstrate the power of our methods, we give an infinite family of totally imaginary quadratic number fields such that $\Aut(\PSL(2,q^2)),$ for $q$ an odd prime power, cannot be realized as an unramified Galois group over $K,$ but its maximal solvable quotient can. To prove this result, we determine the ring structure of the étale cohomology ring $\displaystyle H^* (\Spec \mathcal{O}_K;\ZZ/ 2\ZZ)$  where $\mathcal{O}_K$ is the ring of integers of an arbitrary totally imaginary number field $K.$
\end{abstract}
\maketitle

\section{Introduction}
What is the structure of the absolute Galois group $\Gamma_K$ of a field $K?$ The famous inverse Galois problem approaches this question by asking what finite groups occur as finite quotients of $\Gamma_K $ (see for example \cite{MalleInverse} or \cite{SerreGalois}). In this paper, we use homotopy-theoretical methods to attack a closely related problem, that of embedding problems. We construct new obstructions to the solvability of embedding problems for profinite groups, and these obstructions allow one to study embedding problems with perfect kernel. As a specific example of how to apply these techniques, we give an infinite family of totally imaginary number fields such that for any field $K$ in this family, $\Aut(\PSL(2,q^2)),$ where $q$ an odd prime power, cannot be realized as an unramified Galois group over $K.$ To prove this result we need to determine the ring structure of the étale cohomology ring $\displaystyle H^*(X,\ZZ/2\ZZ)$ for $X = \Spec \mathcal{O}_L$ the ring of integers of an arbitrary totally imaginary number field $L.$ To allow ourselves to state the results more precisely, recall the following definition.
\begin{definition}
Let $\Gamma$ be a profinite group. Then a finite embedding problem $\textbf{E}$ for $\Gamma$ is a diagram $$\begin{tikzcd}   &  \Gamma \arrow[d,"p"]\\  G \arrow[r,"f"] & H \end{tikzcd}$$ where $G,H$ are finite groups, $f$ is surjective and $p$ is continuous and surjective where $H$ is given the discrete topology. We say that the embedding problem has a proper solution if there exists a continuous surjective homomorphism $q:\Gamma \rightarrow G$ such that $p = fq.$  The embedding problem has a weak solution if there exists a continuous map $q: \Gamma \rightarrow G$ (not neccesarily surjective) such that $p=fq.$ We will call $\ker f$ the kernel of the embedding problem $\textbf{E}.$   
\end{definition}
\noindent If $\Gamma = \Gamma_K$ is the absolute Galois group of a field $K$ and $H$ in the above diagram is trivial, then a proper solution to $\textbf{E}$ corresponds to a realization of $G$ as a Galois group over $K.$  Suppose that we have an embedding problem $\textbf{E}$ as above which we suspect has no solutions. One method to prove that no (weak) solutions exist is to note that if the kernel $P = \ker f$ is abelian, then the extension
$$0 \rightarrow P \rightarrow G \xrightarrow{f} H \rightarrow 0$$ is classified by an element $$z \in H^2(H,P).$$ We can pullback $z$ by the map $p: \Gamma \rightarrow H$ and if the embedding problem $\textbf{E}$ has a solution, then $$p^*(z) \in H^2(\Gamma,P)$$ is zero. If however the kernel of $f$ is not abelian, it is less clear how to proceed. One method is as follows: if $P_{ab}$ is the abelianization of $P,$ and $[P,P]$ the commutator subgroup, consider the abelianized embedding problem  $\textbf{E}_{ab}$ given by $$\begin{tikzcd}  & & &  \Gamma \arrow[d,"p"]  \\ 0 \arrow[r] & P_{ab} \arrow[r] & G/[P,P] \arrow[r]  &  H  \arrow[r] & 0.\end{tikzcd}$$  By the above considerations we get an obstruction element $$z \in H^2(\Gamma,P_{ab})$$ to the solvability of $\textbf{E}_{ab}.$  Further, it is clear that if the original embedding problem $\textbf{E}$ is solvable, so is $\textbf{E}_{ab}.$ It thus follows that if $z \neq 0,$ then $\textbf{E}$ is not solvable. This method of using obstructions in $H^2(\Gamma,P_{ab})$ is classical, see for example \cite[Ch.3]{EmbeddingFaddeev}. The relation of this obstruction with the Brauer-Manin obstruction was studied by A. P\'al and the second author in \cite{SchlankBrauer}. There are obvious limitations to this method. For example, if $P$ is perfect, then we get no useful information since $H^2(\Gamma,P_{ab})=0.$  We develop a theory which remedies this situation by producing non-trivial obstructions to the solution of embedding problems $\textbf{E}$ with perfect kernel, i.e. when $P_{ab}=0.$ For any embedding problem with perfect kernel $P$ and any finitely generated abelian group $A$ together with a fixed element $a \in A$ we construct an obstruction $$o_2^a \in H^3(\Gamma,H_2(P,A)).$$ It should be noted that the obstruction $$o_2^a \in H^3(\Gamma,H_2(P,A))$$ is part of a family of higher obstructions $$o_n^a \in H^{n+1}(\Gamma,H_n(P,A))$$  for $n=1,2, \ldots,$ where $o_{n+1}^a$ is defined if $o_n^a$ vanishes. This obstruction generalizes the classical one, in the sense that if $n=1, A = \ZZ$ and $a=1,$ then $H_1(P,\ZZ) = P_{ab},$ and $o_1^1 \in H^2(\Gamma,P_{ab})$ coincides with the classical obstruction $p^\ast
(z)$ outlined above. For $n \geq 2,$ the elements $o_n^a,$ thus really are higher obstructions.  The natural formulation for these higher obstructions is in the language of homotopy theory, as developed in \cite{Barnea-Schlank}. For $n=2,$ it turns out that one can give an elementary construction of these obstructions in terms of crossed modules and we choose this development in this paper, leaving the general formulation of these obstructions to a future paper. \\ \\
We now sketch how we apply the above methods to show that $\Aut(\PSL(2,q^2))$ for $q$ an odd prime power cannot be realized as the Galois group of an unramified Galois extension of certain totally imaginary number fields. Let $K$ be a totally imaginary number field and let $\Gamma_K^{ur}$ be the unramified Galois group of $K.$ Suppose that  $$H^1(\Gamma_K^{ur},\ZZ/2\ZZ) \cong \ZZ/2\ZZ \times \ZZ/2\ZZ$$ and let $a,b$ be any choice of generators. The elements $a$ and $b$  define a surjective map $p_{ab}:\Gamma_K^{ur} \rightarrow \ZZ/2\ZZ \times \ZZ/2\ZZ.$ We then have the embedding problem given by the diagram $$\begin{tikzcd} & & &  \Gamma_K^{ur} \arrow[d,"p_{ab}"] \arrow[dl,dotted,"?"description] \\ 0 \arrow[r] & \PSL(2,q^2) \arrow[r] & \Aut(\PSL(2,q^2)) \arrow[r]  &  \ZZ/2\ZZ \times \ZZ/2\ZZ  \arrow[r] & 0\end{tikzcd}$$ where the map $$\Aut(\PSL(2,q^2)) \rightarrow \Out(\PSL(2,q^2)) \cong \ZZ/2\ZZ \times \ZZ/2\ZZ$$ is the natural quotient map. We see that $\Aut(\PSL(2,q^2))$ occurs as an unramified Galois group over $K$ if and only if for some surjection  $$p_{ab}: \Gamma_K^{ur} \rightarrow \ZZ/2\ZZ \times \ZZ/2\ZZ,$$ corresponding to generators $a,b$ of $H^1(\Gamma_K^{ur},\ZZ/2\ZZ),$ there is a surjective map filling in the dotted arrow. From the above obstructions to solutions to embedding problems with perfect kernel we get for each $p_{ab}$ an obstruction $$o_2 \in H^3(\Gamma_K^{ur},H_2(\PSL(2,q^2),\ZZ/2\ZZ)) \cong H^3(\Gamma_K^{ur},\ZZ/2\ZZ),$$ depending of course on $a$ and $b.$ We shall prove that the obstruction $o_2$ can be identified with the cup product $c^2 \cup d$ where $c$ and $d$ are some two distinct and non-zero linear combinations  of the elements $a,b \in H^1(\Gamma_K^{ur},\ZZ/2\ZZ).$ If the obstruction $o_2$ is non-zero, for each choice of generators of $H^1(\Gamma_K^{ur},\ZZ/2\ZZ),$ then $\Aut(\PSL(2,q^2))$ cannot be realized as an unramified Galois group over $K.$ Letting $X = \Spec \mathcal{O}_K$ be the ring of integers of $K$ and $X_{et}$ be the small étale site of $X$ and $B \Gamma_K^{ur}$ the classifying site of $\Gamma_K^{ur},$  there is a tautological map of sites $$k:X_{et} \rightarrow B\pi_1(X) = B \Gamma_K^{ur}.$$  Further, $k$ induces an isomorphism $$H^1(X,\ZZ/2\ZZ) \cong H^1(\Gamma_K^{ur},\ZZ/2\ZZ),$$ so to show that $o_2$ is non-zero it clearly suffices to prove that $k^*(o_2)$ is non-zero and this is equivalent to the cup product $a^2 \cup b$ of $$a,b \in H^1(X,\ZZ/2\ZZ) \cong \ZZ/2\ZZ \times \ZZ/2\ZZ$$ being non-zero. To prove that this cup product is non-zero, we must study the étale cohomology ring of a totally imaginary number field $X = \Spec \mathcal{O}_K.$ In the beautiful paper \cite{MazurNotes}, Mazur, among other things, determined the étale cohomology groups $H^i(X,\ZZ/n\ZZ)$ using Artin-Verdier duality. We continue the investigation into the structure of the étale cohomology of a number field by explicitly determining the ring $\displaystyle H^*(X,\ZZ/2\ZZ).$ Using our determination of the ring $\displaystyle H^*(X,\ZZ/2\ZZ)$ we can give neccessary and sufficient conditions on the field $K$ for the above cup product to be non-zero and hence give conditions for which $\Aut(\PSL(2,q^2))$ is not the Galois group of an unramified extension of $K.$  These conditions are satisfied for a wide variety of totally imaginary number fields $K,$  a more precise description of these conditions will follow as we now state our main theorems. \ \\ \noindent

The full description of our results involve some notation so as to give the reader a taste of what we prove, we provide some special cases, leaving the most general formulation to the main text. For $x \in H^1(X,\ZZ/2\ZZ)$ we will view $x$ as the ring of integers of a quadratic unramified extension $L=K(\sqrt{c})$ with $c \in K^*.$ Note that $\divis(c)= 2I$ for some fractional ideal of $K$ since $L$ is an unramified extension. For a proof of the following theorem, see Theorem \ref{prop:triplecup}.
\begin{theorem} \label{thm:cupstr}
Let $X = \Spec \mathcal{O}_K$ where $K$ is a totally imaginary number field and let $x$ and $y$ be elements in $H^1(X,\ZZ/2\ZZ),$ corresponding to the unramified quadratic extensions $L= K(\sqrt{c}), M= K(\sqrt{d}),$ where $c,d \in K^*.$  Let $$\divis(d)/2 = \prod_i \mathfrak{p}_i^{e_i}$$ be the factorization of $\divis(d)/2$ into prime ideals in $\mathcal{O}_K.$ Then $$x \cup x \cup y \in H^3(X,\ZZ/2\ZZ)$$ is non-zero if and only if $$\sum_{\mathfrak{p}_i \text{ inert in } L} e_i \equiv 1 \pmod{2}.$$ 
\end{theorem}
\begin{remark}
Christian Maire \cite{MaireUnramified} applied Theorem \ref{thm:cupstr} to verify some special cases of the unramified Fontaine-Mazur conjecture. 
\end{remark}
In the following theorem, if $K$ is a totally imaginary number field note as before that for any quadratic unramified extension $L=K(\sqrt{c})$ that $\divis(c)$ is even, so that $\divis(c)/2$ makes sense. The following theorem is proved in Theorem \ref{thm:nosolution2} 
\begin{theorem} \label{thm:nosolution2}
Let $K$ be a totally imaginary number field and suppose that for any two distinct quadratic extension $L=K(\sqrt{c}), M = K(\sqrt{d})$ where $c,d \in K^* \setminus (K^*)^2,$ that if $$\divis(d)/2 = \prod \mathfrak{p}_i^{e_i}$$ is the prime factorization of $\divis(d)/2,$  $$\sum_{\mathfrak{p}_i \text{ unramified in L} } e_i \equiv 1  \mod 2.$$ Then $\Aut(\PSL(2,q^2))$ cannot be realized as the Galois group of an unramified extension of $K,$ but its maximal solvable quotient $$\Aut(\PSL(2,q^2))^{sol} \cong \ZZ/2\ZZ \oplus \ZZ/2\ZZ$$ can.
\end{theorem} 
\begin{remark}
This theorem is not of maximal generality, but is an example of an application of our methods. There is nothing that stops one from proving a similar theorem, using our methods, for larger classes of groups. 
\end{remark}
Let us note that there are infinitely many totally imaginary number fields for which $\Aut(\PSL(2,q^2))$ cannot be an unramified Galois group. In fact, the following proposition, corresponding Corollary \ref{prop:unramifiedQuadraticCup1} in the main text, shows that there are even infinitely many totally imaginary quadratic number fields with this property. 
\begin{prop} 
Let $p_1,p_2,p_3$ be three primes such that $$p_1 p_2 p_3 \equiv 3 \mod 4,$$  and $$\genfrac(){}{0}{p_i}{p_j} = -1$$ for all $i \neq j.$  Let $K = \mathbb{Q}(\sqrt{-p_1p_2p_3}).$ Then $\Aut(\PSL(2,q^2))$ cannot be realized as the Galois group of an unramified extension of $K,$ but its maximal solvable quotient $\Aut(\PSL(2,q^2))^{sol} \cong \ZZ/2\ZZ \oplus \ZZ/2\ZZ$ can.
\end{prop}

 Little is known about $\Gamma_K^{ur}$ in general. The group $\Gamma_K^{ur}$ can be infinite, for example, Golod and Shafarevich \cite{Golod} showed that there are number fields which have infinite Hilbert class field tower, an example being $$\mathbb{Q}(\sqrt{-4849845}) = \mathbb{Q}(\sqrt{-3\cdot 5 \cdot 7 \cdot 11 \cdot 13 \cdot 17 \cdot 19}).$$ It is not true, however, that $\Gamma_K^{ur}$ has to be solvable (see for example \cite{Maire}).
 In \cite{Yamamura} Yakamura determines the unramified Galois group for all imaginary quadratic fields $K$ of class number $2$ and shows that in this situation $\Gamma_K^{ur}$ is finite. The results of \cite{Yamamura} are unconditional, except for the case $\mathbb{Q}(\sqrt{-427})$ where the generalized Riemann hypothesis is assumed. Yakamura uses discriminant bounds to determine $\Gamma_K^{ur}$ and as such, his methods are of a of quantitative nature. In contrast, the methods in this paper are of a qualitative nature, using algebraic invariants to obstruct the existence of certain groups as unramified Galois groups. This approach has the benefit that it can be applied to totally imaginary number fields of arbitrarily large discriminant.
\subsection{Organization} 
In the first part of Section ~\ref{section:embedding} we study embedding problems for general profinite groups and define a new homotopical obstruction. We then specialize to give two  infinite families of embedding problems which have no solutions. In Section \ref{sec:proppro} we use these obstructions to give examples of infinite families of groups which cannot be realized as unramified Galois group of certain families of totally imaginary number field $K.$ We also state the main results on the cup product structure of the étale cohomology ring $\displaystyle H^*(X,\ZZ/2\ZZ)$ where $X= \Spec \mathcal{O}_K$ is the ring of integers of $K,$ but defer the proofs to Section \ref{section:cohomologyring}. In \ref{section:lemmasonetale} we review the étale cohomology of totally imaginary number fields from Mazur \cite{MazurNotes} (see also \cite{MilneADT}) and Artin-Verdier duality. In this section, we give a description of $\Ext^2_X(\ZZ/n\ZZ,\mathbb{G}_{m,X})$ in terms of number-theoretic data which will be important to us in determining the ring structure of $\displaystyle H^\ast(X,\ZZ/2\ZZ).$ In Section ~\ref{section:cohomologyring} we give a full description of the ring $\displaystyle H^\ast(X,\ZZ/2\ZZ)$ using the results from Section ~\ref{section:lemmasonetale}.
\subsection*{Acknowledgements}
The authors wish to thank Lior Bary-Soroker, Tilman Bauer and Nathaniel Stapleton for their helpful comments.
\section{Higher obstructions for the solvability of embedding problems} \label{section:embedding}
\noindent  In this section we will study embedding problems and obstructions to their solutions. The first goal of this section is to produce non-trivial obstructions to the solution of embedding problems $$\begin{tikzcd}   &  \Gamma \arrow[d,"p"]\\  G \arrow[r,"f"] & H \end{tikzcd}$$ when $\ker f =P$ is perfect. Recall that we assume that $G$ and $H$ are finite groups so that $P$ is finite as well. In the second part of this section, we apply these methods to give examples of two infinite families of embedding problems with no solutions. In the first family, the embedding problems are of the form $$f:G \rightarrow \ZZ/2\ZZ, p:\Gamma \rightarrow \ZZ/2\ZZ$$ where the finite group $G$ satisfies a property we call $(\ast)_a$ and $\Gamma$ satisfies a property we call $(\ast \ast)_a,$  while the second family of embedding problems will be of the form $$f: G \rightarrow \ZZ/2\ZZ \times \ZZ/2\ZZ , p:\Gamma \rightarrow \ZZ/2\ZZ \times \ZZ/2\ZZ$$ where $G$ satisfies the property $(\ast)_b$ and $\Gamma$ the property $(\ast \ast)_b.$ 
 In the following section we produce two infinite family of profinite groups, where each $\Gamma$ in the first family satisfies Property $(\ast \ast)_a$ and for each member of the second family, $(\ast \ast)_b$ is satisfied. In fact, we show that in each situation, one can take $\Gamma$ to be equal to $\Gamma_K^{ur},$ the unramified Galois group of $K,$ for $K$ a totally imaginary number field satisfying certain conditions. We stress that the conditions $K$ must satisfy for $\Gamma_K^{ur}$ to satisfy property $(\ast)_a$  are in general different from the conditions $K$ must satisfy for $\Gamma_K^{ur}$ to have property $(\ast)_b.$ We apply this to show that for imaginary number fields $K$ such that $\Gamma_K^{ur}$ satisfies property $(\ast \ast)_a$ (resp. $(\ast \ast)_b$) groups $G$ satisfying property $(\ast)_a$ (respectively groups $G$ satisfying  $(\ast)_b)$  cannot be the Galois group of an unramified field extension $L/K.$ \\  

\noindent To define the obstructions to embedding problems, we start by studying a closely related problem. Suppose that
\begin{equation} \label{finiteexact} 1 \rightarrow P \rightarrow G \xrightarrow{f} H \rightarrow 1 \end{equation} is an exact sequence of finite groups. We will now produce obstructions to the existence of sections to $f: G \rightarrow H$ when the kernel possibly has no abelian quotients.
\begin{definition}
Let $A$ be a finitely generated abelian group and $P$ a finite group. We say that $P$ is $A$-perfect if $$P_{ab} \otimes_{\mathbb{Z}} A = 0.$$ 
\end{definition}
Clearly if a finite group $P$ is perfect, then $P$ is $A$-perfect for all $A$. If $A$ and $P$ are finite, then examples of $A$-perfect groups $P$ are the ones such that the order of $P_{ab}$ is relatively prime to the order of $A.$ Fix now an element $a \in A.$ Note that since $P$ is $A$-perfect and $A$ is finitely generated, we have by the universal coefficient theorem an isomorphism $H^2(P,H_2(P,A)) \cong \Hom(H_2(P,\ZZ),H_2(P,A)).$ Since $\ZZ$ is a free abelian group, there is a unique homomorphism $s_a:\ZZ \rightarrow A$ which takes $1$ to $a.$ There exists a unique central extension of $P$ by $H_2(P,A)$ which under the isomorphism from the universal coefficient theorem corresponds to the map $H_2(P,\ZZ) \xrightarrow{H_2(P,s_a)} H_2(P,A) .$  Call this extension $E_a$ the $a$-central extension. We now construct an element $$\tilde{o}_2^a \in H^3(P,H_2(P,A))$$  which will be an obstruction to the spltting of $f.$ For defining $\tilde{o}^a_2,$ we introduce the following definition.
\begin{definition}\cite{WhiteheadNote}
A crossed module is a triple $(G,H,d)$ where $G$ and $H$ are groups and $H$ acts on $G$ (we denote this action by $^h g$ for $g \in G$ and $h \in H)$ and $d:G \rightarrow H$ is a homomorphism such that $$d(^hg) =g d(h) g^{-1} ,(g \in G, h \in H)$$ and $$^{d(g)}g' = g g' g^{-1}, (g,g' \in  G).$$
\end{definition}
If $H$ acts on $G$ and we have a map $d:G \rightarrow H$ satisfying the conditions of the above definition, we will sometimes write that $d:G \rightarrow H$ is a crossed module. An example of a crossed module is a normal inclusion $P \xrightarrow{i} G,$ where $G$ acts on $P$ by conjugation. If we have such a normal inclusion, we also see that the action of $G$ on $P$ induces an action of $G$ on $H_2(P,A)$ for any abelian group $A.$ 
\begin{lemma}\label{crossedperfect}
Let $A$ be a finitely generated abelian group, $a \in A,$ and $$\begin{tikzcd} 1 \arrow[r] & P \arrow[r,"i"] & G \arrow[r,,"f"] & H \arrow[r] & 1 \end{tikzcd}$$  a short exact sequence of finite groups such that $P$ is $A$-perfect. Take $$u_a \in H^2(P,H_2(P,A))$$ to be the element classifying the $a$-central extension $$1 \rightarrow H_2(P,A) \rightarrow C \xrightarrow{p} P \rightarrow 1.$$ Then there is a unique action of $G$ on $C$ compatible with the action of $G$ on $P$  and $H_2(P,A)$ such that the map $$i \circ p =d: C \rightarrow G$$ is a crossed module.
\end{lemma}
We will call the crossed module $d: C \rightarrow G$ constructed above for the $(A,a)$-obstruction crossed module associated to $f.$ 
Before proving Lemma \ref{crossedperfect}, some preliminaries on central extensions are needed. If $$E: 1 \rightarrow N \xrightarrow{h} G \xrightarrow{g} P \rightarrow 1$$ is a central extension classified by an element $z_E \in H^2(P,N),$ we let $$\delta^E_*:H_2 (P,\ZZ) \rightarrow N$$ be the image of $z_E$ in $\Hom(H_2(P,\ZZ),N)$ under application of the map $$H^2(P,N) \rightarrow \Hom(H_2(P,\ZZ),N)$$ coming from the universal coefficient theorem. 
\begin{prop}[{\cite[V.6 Prop. ~6.1]{StammbachHomology}}]\label{prop:stammbach}
Suppose that $$E: 1 \rightarrow N \xrightarrow{h} G \xrightarrow{g} P \rightarrow 1$$ and $$E': 1 \rightarrow N' \xrightarrow{h'} G' \xrightarrow{g'} P' \rightarrow 1$$ are two central extensions. Let $r:N' \rightarrow N$ and $s:P \rightarrow P'$ be group homomorphisms and suppose that $\Ext^1(P_{ab},N')=0.$  Then there exists $t:G \rightarrow G'$ inducing $r,s$ if and only if $$\begin{tikzcd} H_2(P,\ZZ) \arrow[r,"\delta^E"] \arrow[d,"s"] & N \arrow[d,"r"] \\ H_2(P',\ZZ) \arrow[r,"\delta^{E'}"] & N' \end{tikzcd}$$ is commutative. If $t$ exists, then it is unique if and only if $\Hom(P_{ab},N')=0.$ 
\end{prop}
\begin{proof}[Proof of Lemma \ref{crossedperfect}]
Since  $$E:1 \rightarrow H_2(P,A) \rightarrow C \xrightarrow{p} P \rightarrow 1$$ is a central extension, the map $C \xrightarrow{p} P$ can be given a canonical crossed module structure. On the other hand, since $P$ is normal in $G,$ there is a crossed module structure on the map $P \xrightarrow{i} G.$ We want to show that there is unique  crossed module structure on the map $d= ip:C \rightarrow G$ compatible with the actions of $G$ on $P$ and $H_2(P,A).$  Let us start by defining an action of $G$ on $C.$  If $g \in G$ write $$g: P \rightarrow P$$ for the automorphism $g$ induces on $P$ and $$g:H_2(P,A) \rightarrow H_2(P,A)$$ for the action of $g$ on the coefficient group $H_2(P,A).$ We then have a diagram $$\begin{tikzcd} 1 \arrow[r] & H_2(P,A) \arrow[r,"i"] \arrow[d,"g"] & C \arrow[r,"p"] \arrow[d,dotted,"?"description] & P \arrow[r] \arrow[d,"g"] & 1 \\  1 \arrow[r] & H_2(P,A) \arrow[r,"i"] & C \arrow[r,"p"] & P \arrow[r] & 1.\end{tikzcd}$$ We now claim that there is a unique fill-in of the dotted arrow to a map $g: C \rightarrow C.$ We note that since $P$ is $A$-perfect and $A$ is finitely generated, this implies that $$\Hom(P_{ab},H_2(P,A)) = \Ext^1(P_{ab},H_2(P,A)) = 0. $$ The existence and uniqueness of the fill-in now follows from Proposition \ref{prop:stammbach} since 
$\delta_* ^E:H_2(P,\ZZ) \rightarrow H_2(P,A)$ is induced by the map $\ZZ \rightarrow A$ taking $1$ to $a.$  We thus get a map $G \rightarrow \Aut(C)$ which is a group homomorphism by uniqueness of the fill-ins.  If $g \in G$ and $c \in C,$ we write $^g c$ for the element we get after acting by $g$ on $c.$ To see that this action makes the map $d: C \rightarrow G$ into a crossed module, it must be verified that $d(^gc) = g  \cdot d(c) \cdot g^{-1}$ and that $^{d(c)}c' = c \cdot c' \cdot c^{-1}.$ For the first property, note that $$d(^gc) = i(^gp(c)) = g \cdot d(c) \cdot g^{-1}$$ by the crossed module structure of $i: P \rightarrow G.$ To show that $$^{d(c)}c' = c c' c^{-1}$$ we once again use Proposition \ref{prop:stammbach} and the fact that conjugation by an element $p \in P$ induces the identity homomorphism in group homology. Thus, there is a unique crossed module structure on the map $d: C \rightarrow G$ compatible with the actions of $G$ on $P$ and $H_2(P,A).$  
\end{proof}
The crossed module $d: C \rightarrow G$ from Lemma \ref{crossedperfect} gives rise to the exact sequence $$1 \rightarrow H_2(P,A) \rightarrow C \xrightarrow{d} G \xrightarrow{f} H \rightarrow 1.$$ By \cite[IV.6]{BrownCohomology} this shows that the crossed module $d: G \rightarrow C$ is classified by an element $\tilde{o}_2 \in H^3(H,H_2(P,A)).$ One shows immediately that if there is a section of the map $f: G \rightarrow H,$ then $\tilde{o}_2 =0.$ We have thus proved the following proposition.
\begin{prop}\label{prop:obstructioncrossed}
Suppose that $f:G \rightarrow H$ is a surjective homomorphism of finite groups with kernel $P,$ and that $A$ is an abelian group together with a fixed element $a \in A.$ Let $$\tilde{o}^a_2 \in H^3(H,H_2(P,A))$$ be the element classifying the $(A,a)$-obstruction crossed module  $d: C \rightarrow P$ given by Lemma \ref{crossedperfect}. Then if $\tilde{o}_2^a  \neq 0,$ there is no section of $f: G \rightarrow H.$
\end{prop}
We now finally apply this to embedding problems. Let $\Gamma$ be a profinite group and suppose we have an embedding problem $$\begin{tikzcd} & & &  \Gamma \arrow[d,"p"] \arrow[dl,dotted,"?"description] \\ 1 \arrow[r] & \ker f \arrow[r] & G \arrow[r,"f"] & H \arrow[r] & 1 \end{tikzcd}$$ where $G$ and $H$ are finite. By taking pullbacks, we get the diagram$$\begin{tikzcd} 1 \arrow[r] & P \arrow[r] \arrow[d, equals] & \Gamma \times_{H} G \arrow[r] \arrow[d] &  \Gamma \arrow[d,"p"] \arrow[dl,dotted,"?"description] \arrow[r] & 1 \\ 1 \arrow[r] & P \arrow[r] & G \arrow[r,"f"] & H \arrow[r] & 1. \end{tikzcd}$$  Let us note that the existence of a morphism filling in the dotted arrow is equivalent to the existence of a section of the map $$\Gamma \times_{H} G \rightarrow \Gamma.$$
\begin{corollary}\label{cor:obstructionprofinite}
Let $\Gamma$ be a profinite group and suppose we have an embedding problem $$\begin{tikzcd} & & &  \Gamma \arrow[d,"p"] \arrow[dl,dotted,"?"description] \\ 1 \arrow[r] & \ker f = P \arrow[r] & G \arrow[r,"f"] & H \arrow[r] & 1 \end{tikzcd}$$ where $G$ and $H$ are finite. Let $A$ be an abelian group together with a fixed element $a$ and suppose that $P$ is $A$-perfect and denote by $$ o_2^a = p^*(\tilde{o}^a_2) \in H^3(\Gamma,H_2(P,A))$$ the pullback by $p$ of the element $\tilde{o}^a_2 \in H^3(H,H_2(P,A))$ from Proposition \ref{prop:obstructioncrossed}. Then if $o_2^a \neq 0,$ there are no solutions to the embedding problem.
\end{corollary}
\begin{proof}
By the above discussion, it is enough to show that if there is a section of the map $$\Gamma \times_H G \rightarrow \Gamma$$  then $o_2=0.$ We have the diagram $$\begin{tikzcd}[column sep = small] & & & & \Gamma \arrow[d,"p"] \\ 1 \arrow[r] & H_2(P,A) \arrow[r]&  C \arrow[r] & G \arrow[r] & H \arrow[r] & 1 \end{tikzcd}$$ which we can pullback by $p$ to get a profinite crossed module$$1 \rightarrow H_2(P,A) \rightarrow \Gamma \times_{H} C \rightarrow \Gamma \times_{H} G \rightarrow \Gamma \rightarrow 1.$$ This profinite crossed module is classified by an element $o_2^a \in H^3(\Gamma,H_2(P,A))$ (see for example \cite[pg. 168]{PorterCrossed}) which is the pullback by $p$ of the element $$\tilde{o}^a_2 \in H^3(H,H_2(P,A)).$$ If there is a section of the map $\Gamma \times_H G \rightarrow \Gamma$, this element $o_2^a$ is clearly trivial, so our proposition follows.
\end{proof}
We now restrict Proposition \ref{prop:obstructioncrossed} to when $A = \ZZ/2\ZZ$ and $a=1.$ Instead of writing the $(\ZZ/2\ZZ,1)$-obstruction in what follows, we simply write the $\ZZ/2\ZZ$-obstruction, leaving the $1$ implicit. The followng two propositions describe two cases where the $\ZZ/2\ZZ$-obstruction is non-zero. 
\begin{prop} \label{obstructionprop}
Let $$p: \Gamma \rightarrow \ZZ/2\ZZ, f:G \rightarrow \ZZ/2\ZZ$$ be an embedding problem such that $P = \ker f$ is $\ZZ/2\ZZ$-perfect and let $$\tilde{o}_2 \in H^3(\ZZ/2\ZZ,H_2(P,\ZZ/2\ZZ))$$ be the obstruction to the existence of a section of $f$ given by \ref{prop:obstructioncrossed} and let $x \in H^1(\Gamma,\ZZ/2\ZZ)$ be the class that classifies the map $$p: \Gamma \rightarrow \ZZ/2\ZZ.$$ Then if $$x \cup x \cup x \in H^3(\Gamma,\ZZ/2\ZZ)$$ and $\tilde{o}_2$ are non-zero, there are no solutions to the embedding problem.
\end{prop}
\begin{prop} \label{obstructionprop2}
Let $$p: \Gamma \rightarrow \ZZ/2\ZZ \times \ZZ/2\ZZ, f:G \rightarrow \ZZ/2\ZZ \times \ZZ/2\ZZ $$ be an embedding problem such that $P = \ker f$ is $\ZZ/2\ZZ$-perfect, and such that $H_2(P,\ZZ/2\ZZ) = \ZZ/2\ZZ.$  Let $$\tilde{o}_2 \in H^3(\ZZ/2\ZZ \times \ZZ/2\ZZ,H_2(P,\ZZ/2\ZZ)) \cong H^3(\ZZ/2\ZZ \times \ZZ/2\ZZ, \ZZ/2\ZZ) $$ be the obstruction to the existence of a section of $f$ given by Proposition \ref{prop:obstructioncrossed}. Suppose that $\tilde{o}_2 = a^2 \cup b$ for $a,b \in H^1(\ZZ/2\ZZ \times \ZZ/2\ZZ, \ZZ/2\ZZ)$ distinct and non-zero and let $x_1,x_2 \in H^1(\Gamma,\ZZ/2\ZZ)$ be the pullbacks  of $a$ and $b$ respectively by $p.$ Then if $$x_1^2  \cup x_2 \in H^3(\Gamma,\ZZ/2\ZZ)$$ is non-zero, there are no solutions to the embedding problem.
\end{prop}
The following lemma is used in the proof of Proposition \ref{obstructionprop}.
\begin{lemma} \label{lemma:injectivereduce}
Let $M$ be a $\ZZ/2\ZZ$-module and $i \geq 0$ be odd. For every non-zero $x \in H^i(\ZZ/2\ZZ,M)$ there exists a $\ZZ/2\ZZ$-equivariant map $$\pi:M \rightarrow \ZZ/2\ZZ$$ such that  if $$\pi_*: H^i(\ZZ/2\ZZ,M) \rightarrow H^i(\ZZ/2\ZZ,\ZZ/2\ZZ)$$ is the induced map, then $\pi_*(x) \neq 0.$ 
\end{lemma}
\begin{proof}
Recall that we have a free resolution of $\ZZ$ considered as a module over $\ZZ[\ZZ/2\ZZ]$ that is periodic of order two. Using this resolution we see that for any $\ZZ/2\ZZ$-module $M,$ 
$H^i(\ZZ/2\ZZ,M)$ for $i$ odd can be identified with the crossed homomorphisms $f: \ZZ/2\ZZ \rightarrow M$ modulo principal crossed homomorphisms. Denote by $M_{\ZZ/2\ZZ}$ the coinvariants of $M$ and by $I \subset \ZZ[\ZZ/2\ZZ]$ the augmentation ideal. Note that any crossed homomorphism $f: \ZZ/2\ZZ \rightarrow IM$ becomes a principal crossed homomorphism after composition with the inclusion $IM \rightarrow M.$ This implies that the map $$H^i(\ZZ/2\ZZ,IM) \rightarrow H^i(\ZZ/2\ZZ,M)$$ is zero, so that by exactness, the map $$p:H^i(\ZZ/2\ZZ,M) \rightarrow H^i(\ZZ/2\ZZ,M_{\ZZ/2\ZZ})$$ is injective. Suppose now that $x \in H^i(\ZZ/2\ZZ,M).$ By what we just have shown, to prove our lemma, we can reduce to the case where $M$ has a trivial $\ZZ/2\ZZ$-action. But in such a case the lemma is trivial.
\end{proof}
\begin{proof}[Proof of \ref{obstructionprop} and \ref{obstructionprop2}]
We will start by proving Proposition \ref{obstructionprop} and then prove Proposition \ref{obstructionprop2}. By assumption the obstruction $$\tilde{o}_2 \in H^3(\ZZ/2\ZZ,H_2(P,\ZZ/2\ZZ))$$ is non-zero. By Lemma \ref{lemma:injectivereduce} we can find a $\ZZ/2\ZZ$-equivariant map $\pi:H_2(P,\ZZ/2\ZZ) \rightarrow \ZZ/2\ZZ $ such that if $$\pi_*: H^3(\ZZ/2\ZZ,H_2(P,\ZZ/2\ZZ)) \rightarrow H^3(\ZZ/2\ZZ,\ZZ/2\ZZ)$$ is the induced map, then $\pi_*(\tilde{o}_2) \neq 0.$  We then see that $\pi_*(\tilde{o}_2)$ is the triple cup product of a generator of $H^1(\ZZ/2\ZZ ,\ZZ/2\ZZ).$ By pulling back $\tilde{o}_2$ by $p: \Gamma \rightarrow \ZZ/2\ZZ$ Corollary \ref{cor:obstructionprofinite} gives us an obstruction $$o_2 = p^*(\tilde{o}_2) \in H^3(\Gamma,H_2(P,\ZZ/2\ZZ))$$  which we claim is non-zero. The commutative diagram  $$\begin{tikzcd} H^3(\Gamma,H_2(P,\ZZ/2\ZZ))  \arrow[r,"\pi_*"]   & H^3(\Gamma,\ZZ/2\ZZ)  \\  H^3(\ZZ/2\ZZ,H_2(P,\ZZ/2\ZZ))  \arrow[r,"\pi_*"]  \arrow[u,"p^*"]  & H^3(\ZZ/2\ZZ,\ZZ/2\ZZ)  \arrow[u,"p^*"]  \end{tikzcd}$$  immediately gives that we are reduced to showing that $$\pi_*(o_2) = p^*(\pi_*(\tilde{o}_2)) \in H^3(\Gamma,\ZZ/2\ZZ)$$ is non-zero. We know that $\pi_*(\tilde{o}_2)$ is the triple cup product of the generator of $H^1(\ZZ/2\ZZ, \ZZ/2\ZZ).$ This observation together with the fact that  $$p^*:H^*(\ZZ/2\ZZ,\ZZ/2\ZZ) \rightarrow H^*(\Gamma,\ZZ/2\ZZ)$$ is a ring homomorphism finishes the proof of Proposition \ref{obstructionprop}, since by assumption the triple cup product of the element $x \in H^1(\Gamma,\ZZ/2\ZZ)$ is non-zero. To prove Proposition \ref{obstructionprop2}, we have by assumption that $H_2(P,\ZZ/2\ZZ) = \ZZ/2\ZZ$ and that $\tilde{o}_2 =  a^2 \cup b$ for $a,b \in H^1(\ZZ/2\ZZ \times \ZZ/2\ZZ, \ZZ/2\ZZ)$ distinct and non-zero. What remains is thus to show that $o_2 = p^*(\tilde{o}_2)$ is non-zero. This is guaranteed by the assumption that $x_1^2 \cup x_2 \neq 0. $
\end{proof}
The above propositions motivates the following two conditions on a finite group $G.$ 
\begin{definition}\label{def:groupA}
Let $G$ be a finite group. We say that $G$ has property $(\ast)_a$  if: \begin{enumerate}
\item There exists a surjective map $f:G \rightarrow \ZZ/2\ZZ.$ Denote by $P$ the kernel of $f.$ 
\item The order of $P_{ab}$ is odd.
\item The canonical $\ZZ/2\ZZ$-obstruction $\tilde{o}_2 \in H^3(\ZZ/2\ZZ,H_2(P,\ZZ/2\ZZ))$ is non-zero.
\end{enumerate}
\end{definition}
\begin{definition} \label{def:groupB}
Let $G$ be a finite group. We say that $G$ has property $(\ast)_b$ if: \begin{enumerate}
\item There exists a surjective map $f:G \rightarrow \ZZ/2\ZZ \times \ZZ/2\ZZ$ Denote by $P$ the kernel of $f.$ 
\item The order of $P_{ab}$ is odd and $H_2(P,\ZZ/2\ZZ) = \ZZ/2\ZZ.$ 
\item The canonical $\ZZ/2\ZZ$-obstruction $$\tilde{o}_2 \in H^3(\ZZ/2\ZZ \times \ZZ/2\ZZ,\ZZ/2\ZZ)$$ is equal to $a^2 \cup b$ for $a,b \in H^1(\ZZ/2\ZZ \times \ZZ/2\ZZ , \ZZ/2\ZZ)$ non-zero and distinct. 
\end{enumerate}
\end{definition}
These two properties of a finite group are matched by the following two conditions on a profinite group.
\begin{definition}\label{def:profiniteA}
Let $\Gamma$ be a profinite group. We say that $\Gamma$ has property $(\ast \ast)_a$ if for any surjection $\Gamma \rightarrow \ZZ/2\ZZ$ classified by $a \in H^1(\Gamma,\ZZ/2\ZZ),$ the triple cup product $$a^3\in H^3(\Gamma,\ZZ/2\ZZ)$$ is non-zero.
\end{definition}
\begin{definition}\label{def:profiniteB}
Let $\Gamma$ be a profinite group. We say that $\Gamma$ has property $(\ast \ast)_b$ if for any two distinct surjections $f,g:\Gamma \rightarrow \ZZ/2\ZZ$ classified by $a,b \in H^1(\Gamma,\ZZ/2\ZZ)$ respectively, the cup product $$a^2 \cup b\in H^3(\Gamma,\ZZ/2\ZZ)$$ is non-zero.
\end{definition}
We then have the following two corollaries, The first follows immediately from Proposition \ref{obstructionprop} and the second from Proposition \ref{obstructionprop2}.
\begin{corollary} \label{cor:nosurjections} 
Let $G$ have property $(\ast)_a$ and $\Gamma$ have property $(\ast \ast)_a.$ Then there are no continuous surjections $\Gamma \rightarrow G.$ 
\end{corollary}
\begin{corollary} \label{cor:nosurjections} 
Let $G$ have property $(\ast)_b$ and $\Gamma$ have property $(\ast \ast)_b.$ Then there are no continuous surjections $\Gamma \rightarrow G.$ 
\end{corollary}
\subsection{Two infinite families of groups \nopunct}
In this subsection we will start by showing that there is an infinite family of finite groups $G$ satisfying property $(\ast)_a.$ As a corollary of this, we will then produce an infinite family of groups satisfying property $(\ast)_b.$ 
Take now $q=p^m$ to be an odd prime power and $\alpha \in \Gal(\mathbb{F}_{q^2}/\mathbb{F}_q)$ to be the generator. Let $\PSL(2,q^2)$ be the projective special linear group over $\mathbb{F}_{q^2}$ and consider $$\Aut(\PSL(2,q^2)) \cong \PGL(2,q^2) \rtimes \ZZ/2\ZZ$$ where $\ZZ/2\ZZ$ acts on $\PGL(2,q^2)$ by $\alpha.$ There are three subgroups of $\PGL(2,q^2) \rtimes \ZZ/2\ZZ$ of index two that contains $\PSL(2,q^2).$ Two of these groups correspond to $\PGL(2,q^2)$ and $$\PSL(2,q^2) \rtimes \ZZ/2\ZZ$$ respectively. The third subgroup of index $2$ is traditionally denoted by $M(q^2)$ and is the one of interest to us. When $q=3,$ $M(q^2)$ is known as $M_{10},$ the Mathieu group of degree $10.$ Concretely, $$M(q^2) = \PSL(2,q^2) \cup \alpha \tau \PSL(2,q^2)$$ where $\tau \in \PGL(2,q^2) \setminus \PSL(2,q^2).$ 
\begin{prop}\label{prop:mproperty}
Let $q=p^m$ be an odd prime power. Then $M(q^2)$ as defined above satisfies property $(\ast)_a.$ 
\end{prop}
\begin{proof}
By the above discussion, we have a short exact sequence $$1 \rightarrow \PSL(2,q^2) \xrightarrow{i} M(q^2) \rightarrow \ZZ/2\ZZ \rightarrow 1.$$  We know that \cite[Table 4.1, pg. 302]{GorensteinFinite} $$H_2(\PSL(2,q^2),\ZZ/2\ZZ) \cong \ZZ/2\ZZ$$ and that the non-trivial dual class is realized by the central extension $$1 \rightarrow \ZZ/2\ZZ \rightarrow \SL(2,q^2) \xrightarrow{p} \PSL(2,q^2) \rightarrow 1.$$  We then derive the crossed module $$d=i \circ p: \SL(2,q^2) \rightarrow M(q^2)$$ which gives the exact sequence $$1 \rightarrow \ZZ/2\ZZ \rightarrow \SL(2,q^2) \xrightarrow{d} M(q^2) \rightarrow \ZZ/2\ZZ \rightarrow 1.$$ Here $M(q^2) = \PSL(2,q^2) \cup \alpha \tau \PSL(2,q^2)$ acts on $\SL(2,q^2)$ in the obvious way. To find the element classifying this crossed module, we will find an explicit $3$-cocycle $c: \ZZ/2\ZZ \times \ZZ/2\ZZ \times \ZZ/2\ZZ \rightarrow \ZZ/2\ZZ$ whose cohomology class in $H^3(\ZZ/2\ZZ,\ZZ/2\ZZ)$ corresponds to our crossed module. To start with, we choose a set-theoretic section $s: \ZZ/2\ZZ \rightarrow M(q^2),$ for example, $s(0) = I$ and $s(1) = \alpha \tau$ where $\tau \in \PGL(2,q^2) \setminus \PSL(2,q^2).$ If $\theta$ is a primitive element of $\mathbb{F}_{q^2}$ one  can take $\tau$ to be the equivalence class in $\PGL(2,q^2)$ of the matrix $$\tau = \begin{bmatrix} 1 & 0 \\ 0 &  \theta \end{bmatrix}.$$ The failure of $s$ to be a group homomorphism is measured by a function $$F: \ZZ/2\ZZ \times \ZZ/2\ZZ \rightarrow \ker d$$ that is, $F$ satisfies $$s(i)s(j) = F(i,j) s(i+j)$$ for $i,j \in \ZZ/2\ZZ.$ The only non-zero value of $F$ is when $i=j=1,$  and with the explicit choice of $\tau$ as above,  $F(1,1)$ can be taken to be the matrix $$\begin{bmatrix}\theta^{-(q+1)/2} & 0 \\0 & \theta^{(q+1)/2} \end{bmatrix}$$ in $\PSL(2,q^2).$ We can lift $F$ to a function $$\tilde{F} : \ZZ/2\ZZ \times \ZZ/2\ZZ \rightarrow \SL(2,q^2) $$ by letting $$\tilde{F}(1,1) = \begin{bmatrix}  \theta^{-(q+1)/2} & 0 \\ 0 & \theta^{(q+1)/2} \end{bmatrix} $$ and $\tilde{F}$ the identity otherwise. Let $$c:\ZZ/2\ZZ \times \ZZ/2\ZZ \times \ZZ/2\ZZ \rightarrow \ZZ/2\ZZ$$ be such that for $g,h,k \in \ZZ/2\ZZ,$ $$^{s(g)}\tilde{F}(h,k) \tilde{F}(g,hk) = i(c(g,h,k)) F(g,h)F(gh,k).$$ One once again checks that $c(g,h,k)$ is zero unless $$g=h=k=1,$$ and in this case $c(1,1,1) = 1.$ So $c:\ZZ/2\ZZ \times \ZZ/2\ZZ \times \ZZ/2\ZZ \rightarrow \ZZ/2\ZZ$ gives us a cocycle which in cohomology corresponds to the triple cup product of a generator of $H^1(\ZZ/2\ZZ,\ZZ/2\ZZ).$  By \cite[IV.5]{BrownCohomology} this element classifies the corresponding crossed module, so our proposition follows.
\end{proof}
\begin{prop} \label{prop:autproperty}
Let $q=p^m$ be an odd prime power. Then $\Aut(\PSL(2,q^2))$ satisfies property $(\ast)_b.$ 
\end{prop}
\begin{proof}
We have a short exact sequence $$ 1 \rightarrow \PSL(2,q^2) \rightarrow \Aut(\PSL(2,q^2)) \rightarrow \ZZ/2\ZZ \times \ZZ/2\ZZ \rightarrow 1.$$ As in the proof of \ref{prop:mproperty}, we derive from this exact sequence the crossed module $$1 \rightarrow \ZZ/2\ZZ \rightarrow \SL(2,q^2) \rightarrow \Aut(\PSL(2,q^2)) \rightarrow \ZZ/2\ZZ \times \ZZ/2\ZZ \rightarrow 1.$$ Further this crossed module corresponds to an element $$x \in H^3(\ZZ/2\ZZ \times \ZZ/2\ZZ,\ZZ/2\ZZ).$$ This group $H^3(\ZZ/2\ZZ \times \ZZ/2\ZZ,\ZZ/2\ZZ)$ is $4$-dimensional as a $\ZZ/2\ZZ$-vector space, spanned by $$a^3,a^2 \cup b, a \cup b^2, b^3$$ for $a,b \in H^1(\ZZ/2\ZZ \times \ZZ/2\ZZ, \ZZ/2\ZZ)$ generators. From this crossed module we get three different crossed modules by pullback along different maps $\ZZ/2\ZZ \rightarrow \ZZ/2\ZZ \times \ZZ/2\ZZ:$ the first two by pullback along the inclusions $\ZZ/2\ZZ \rightarrow \ZZ/2\ZZ \times \ZZ/2\ZZ$ of the $i$th factor and the third by pullback along the diagonal map $\Delta:\ZZ/2\ZZ \rightarrow \ZZ/2\ZZ \times \ZZ/2\ZZ.$ Since for the first two crossed modules the projection map onto $\ZZ/2\ZZ$ has a section, they are equivalent to the trivial crossed module and thus represent $0$ in $H^3(\ZZ/2\ZZ,\ZZ/2\ZZ).$ The third crossed module is the crossed module occuring in the proof of Proposition \ref{prop:mproperty}. This implies that $x$ is either equal to $a^2 \cup b$ or $a \cup b^2.$  But this shows that $\Aut(\PSL(2,q^2))$ satisfies property $(\ast)_b.$
\end{proof}
\section{Obstructions to the unramified inverse Galois problem} \label{sec:proppro}
In this section we will show that groups that satisfy property $(\ast)_a$ (respectively $(\ast)_b$) cannot be unramified Galois groups when $K$ is a totally imaginary number field such that $\Gamma_K^{ur}$ satisfies property $(\ast \ast)_a$ (respectively $(\ast \ast)_b).$ 
 To explain notation in the following theorem, note that if $L/K$ is an unramified quadratic extension, then if we write $L = K(\sqrt{c}),$ $\divis(c)$ must be an even divisor. Thus it makes sense to write $\divis(c)/2.$
\begin{theorem} \label{thm:nosolution}
Let $K$ be a totally imaginary number field and $G$ be a finite group satisfying property $(\ast)_a$ (for example $M(q^2)).$ Suppose that for each unramified quadratic extension $L=K(\sqrt{c}),$ where $c \in K^* \setminus (K^*)^2,$ that if $$\divis(c)/2  = \prod \mathfrak{p}_i^{e_i}$$ is the prime factorization of $\divis(c)/2,$ then $$\sum_{\mathfrak{p}_i \text{ unramified in L} } e_i \equiv 1  \mod 2.$$ Then there does not exist an unramified Galois extension $M/K$ with $G= \Gal(M/K).$ 
\end{theorem} 
\begin{theorem} \label{thm:nosolution2}
Let $K$ be a totally imaginary number field and $G$ be a finite group satisfying property $(\ast)_b$ (for example $\Aut(\PSL(2,q^2)).$ Suppose that for any two distinct quadratic extension $L=K(\sqrt{c}), M = K(\sqrt{d})$ where $c,d \in K^* \setminus (K^*)^2,$ that if $$\divis(d)/2 = \prod \mathfrak{p}_i^{e_i}$$ is the prime factorization of $\divis(d)/2,$ then $$\sum_{\mathfrak{p}_i \text{ unramified in L} } e_i \equiv 1  \mod 2.$$ Then there does not exist an unramified Galois extension $M/K$ with $G= \Gal(M/K).$ 
\end{theorem} 
We now give two infinite families of imaginary quadratic number fields, such that for the first family, no group satisfying property $(\ast)_a$ occurs as an unramified Galois group, while for the latter family, no group satisfying $(\ast)_b$ can be realized as an unramified Galois group.
\begin{corollary}\label{prop:unramifiedQuadraticCup}
Let $p$ and $q$ be two primes such that $$p \equiv 1\mod 4,$$ $$q \equiv 3 \mod 4$$ and $$\genfrac(){}{0}{q}{p} = -1,$$ $K = \mathbb{Q}(\sqrt{-pq})$  and $G$ a finite group satisfying property $(\ast)$ (for example $M(q^2)).$  Then there does not exist an unramified Galois extension $L/K$ with $G= \Gal(L/K),$ but the  maximal solvable quotient $G^{solv} \cong \ZZ/2\ZZ$ is realizable as an unramified Galois group over $K.$  
\end{corollary}
\begin{proof}
All we need to show is that the conditions of Theorem \ref{thm:nosolution} are satisfied. By \cite{SpearmanUnramified}, there is a unique unramified quadratic extension of $K,$ given by $L = K(\sqrt{p}) = \mathbb{Q}(\sqrt{p},\sqrt{-q}).$ We must prove that  $\divis(p)/2 = \mathfrak{p}$ is inert in $L.$  One easily sees that this is the same as saying that $p$ is inert in $\mathbb{Q}(\sqrt{-q}),$ i.e that $p$ is not a square mod $q,$ which is guaranteed by our assumptions.
\end{proof}
\begin{corollary}\label{prop:unramifiedQuadraticCup1}
Let $p_1,p_2,p_3$ be three primes such that $$p_1 p_2 p_3 \equiv 3 \mod 4,$$  and $$\genfrac(){}{0}{p_i}{p_j} = -1$$ for all $i \neq j.$  Let $K = \mathbb{Q}(\sqrt{-p_1p_2p_3})$  and $G$ be a finite group satisfying property $(\ast)_b$ (for example $\Aut(\PSL(2,q^2)).$  Then there does not exist an unramified Galois extension $L/K$ with $G= \Gal(L/K),$ but the maximal solvable quotient $G^{solv} \cong \ZZ/2\ZZ \oplus \ZZ/2\ZZ$ is realizable as an unramified Galois group over $K.$ 
\end{corollary}
\begin{proof}
For notational purposes, if $p$ is a prime, we let $p^* = (-1)^{(p-1)/2}p.$ We must show that the conditions in Theorem \ref{thm:nosolution2} are satisfied. We proceed as in Corollary \ref{prop:unramifiedQuadraticCup}. Note that by \cite{SpearmanUnramified}, the unramified quadratic extensions of $K$ are given by adjoining a square root of $p_i^*$ for $i=1,2,3.$ Thus given two distinct unramified quadratic extensions $L = K(\sqrt{p_i^*}), M = K(\sqrt{p_j^*}),$ we must show that $\divis(p_j)/2 = \mathfrak{q}_j$ is inert in $L.$ One easily shows that this is the same as saying that $p_j$ is inert in $\mathbb{Q}(\sqrt{p_i^*}),$ and this follows from our congruence conditions.
\end{proof}
For proving Theorem \ref{thm:nosolution} and Theorem \ref{thm:nosolution2} we will need the following proposition, which we prove in Section \ref{section:cohomologyring}. Before stating it, recall that if we have a scheme $X$ and an element $$x \in H^1(X,\ZZ/2\ZZ)$$ then $x$ can be represented by a $\ZZ/2\ZZ$-torsor $p:Y \rightarrow X.$ If $$X = \Spec \mathcal{O}_K$$ is the ring of integers of a number field $K,$ then such a $\ZZ/2\ZZ$-torsor $p:Y \rightarrow X$ can be represented by a scheme $Y = \Spec \mathcal{O}_L$ where $$L= K(\sqrt{c}), \ \ c \in K^* \setminus (K^*)^2$$ and $L/K$ is unramified. 
\begin{restatable}{prop}{triplecup} \label{prop:triplecup}
Let $X = \Spec \mathcal{O}_K$ where $K$ is a totally imaginary number field and let $x$ and $y$ be elements in $H^1(X,\ZZ/2\ZZ),$ corresponding to the unramified quadratic extensions $L= K(\sqrt{c}), M= K(\sqrt{d}),$ where $c,d \in K^*.$  Let $$\divis(d)/2 = \prod_i \mathfrak{p}_i^{e_i}$$ be the factorization of $\divis(d)/2$ into prime ideals in $\mathcal{O}_K.$ Then $x \cup x \cup y \in H^3(X,\ZZ/2\ZZ)$ is non-zero if and only if $$\sum_{\mathfrak{p}_i \text{ inert in } L} e_i \equiv 1 \pmod{2}.$$ 
\end{restatable}
\begin{proof}[Proof of Theorem \ref{thm:nosolution} and Theorem \ref{thm:nosolution2}]
We prove Theorem \ref{thm:nosolution}, the proof of Theorem \ref{thm:nosolution2} uses exactly the same methods. We need to show, by \ref{cor:nosurjections}, that $\Gamma_K^{ur}$ satisfies property $(\ast \ast)_a.$ We prove thus that $$a \cup a \cup a \neq 0$$ for any $a \in H^1(\Gamma_K^{ur},\ZZ/2\ZZ).$ Let $X = \Spec \mathcal{O}_K$ and consider the canonical geometric morphism $$k:X_{et} \rightarrow B \pi_1(X) = B \Gamma_K^{ur}$$ between the étale site of $X$ and the classifying site of $\Gamma_K^{ur}.$ Since $$H^1(\Gamma_K^{ur},\ZZ/2\ZZ) \cong H^1(X_{et},\ZZ/2\ZZ)$$ we see that for $a \cup a \cup a \neq 0,$ it is enough that $k^*(a \cup a \cup a )$ is non-zero.  By $k^*$ defining a ring homomorphism, $k^*(a \cup a \cup a)$ is the same as the triple cup product of a non-zero element of $H^1(X,\ZZ/2\ZZ),$ which we know is non-zero by Proposition \ref{prop:triplecup}. 
\end{proof}
In proving Theorem \ref{thm:nosolution} and  Theorem \ref{thm:nosolution2} above, we used Proposition \ref{prop:triplecup} in a crucial way. The following sections are dedicated to determining the ring structure of $H^*(X,\ZZ/2\ZZ)$ where $X = \Spec \mathcal{O}_K.$ 
\section{The cohomology groups of a totally imaginary number field}\label{section:lemmasonetale}
\noindent In the remarkable paper \cite{MazurNotes}, Mazur investigated the étale cohomology of number fields and proved several seminal results. We start by recalling some of these results.  Let $$X = \Spec \mathcal{O}_K$$ be the ring of integers of a totally imaginary number field $K, \mathbb{G}_{m,X}$ the sheaf of units on $X$ and $F$ be any constructible sheaf. Denote by $^\sim$ the functor $$\RHom_{\Ab}(-,\mathbb{Q}/\ZZ): \Drv(\Ab)^{op} \rightarrow \Drv(\Ab). $$ We now let $$A:R\Gamma(X,F) \rightarrow \RHom_X(\ZZ/2\ZZ_X,\mathbb{G}_{m,X})^\sim[3]$$ be the map adjoint to the composition $$\RHom_X(\ZZ/2\ZZ_X,\mathbb{G}_{m,X}) \times \RHom_X(\ZZ_X,\ZZ/2\ZZ_X) \rightarrow \RHom_X(\ZZ_X,\mathbb{G}_{m,X})$$ followed by the trace map $\RHom_X(\ZZ_X,\mathbb{G}_{m,X}) \rightarrow \mathbb{Q}/\ZZ[-3].$  Artin-Verdier duality then immediately shows that $A$ is an isomorphism in $\Drv(\Ab).$  This pairing satisfies the following compatability condition.
\begin{lemma}\label{dualcup}
Let $X = \Spec \mathcal{O}_K$ be the ring of integers of a totally imaginary number field and $f: F \rightarrow G$ be a morphism in $\Drv(X_{et})$ between complexes of constructible sheaves. Then the map $$R\Gamma(X,F) \xrightarrow{R\Gamma(X,f)} R\Gamma(X,G)$$ in $\Drv(\Ab),$ corresponds under Artin-Verdier duality to $$\RHom_X(F,\mathbb{G}_{m,X})^\sim[3] \xrightarrow{\RHom_X(f,\mathbb{G}_{m,X})^\sim[3]} \RHom_X(G,\mathbb{G}_{m,X})^\sim[3].$$
\end{lemma}
\begin{proof}
Our claim is that the diagram $$\begin{tikzcd}[column sep= 15ex] R\Gamma(X,F) \arrow[r,"{R\Gamma(X,f)}"] \arrow[d,"A"] &  R\Gamma(X,G)  \arrow[d,"A"] \\ \RHom_X(F,\mathbb{G}_{m,X})^\sim[3] \arrow[r,"{\RHom_X(f,\mathbb{G}_{m,X})}"] & \RHom_X(G,\mathbb{G}_{m,X})^\sim[3] \end{tikzcd}$$ is commutative. This follows from the fact that the following diagram is commutative in the obvious way$$\begin{tikzcd}[column sep = small]  \arrow[d,shift left = 10ex,"{\RHom_X(\ZZ_X, f)}",swap]\RHom_X(F,\mathbb{G}_{m,X})  \times  \RHom_X(\ZZ_X,F)  \arrow[r] & \RHom_X(\ZZ_X, \mathbb{G}_{m,X})  \arrow[d,equals] \\ \RHom_X(G,\mathbb{G}_{m,X}) \times \RHom_X(\ZZ_X,G) \arrow[u,shift right = -10 ex,"{\RHom_X(f,\mathbb{G}_{m,X})}"]  \arrow[r] & \RHom_X(\ZZ_X,\mathbb{G}_{m,X}) \end{tikzcd}$$ where the horizontal arrows are given by composition.
\end{proof}
Using the Artin-Verdier pairing one can calculate $H^i(X,\ZZ/n\ZZ)$ by first determining $\Ext^i_X(\ZZ/n\ZZ,\mathbb{G}_{m,X})$ (see \cite{MazurNotes}, but also Lemma \ref{cor:extdescrip}) and then by duality we get
$$H^i(X,\ZZ/n\ZZ) = \begin{cases} \ZZ/n\ZZ & \mbox{ if } i= 0 \\ 
 (\Pic(X)/n)^\sim & \mbox{ if } i= 1 \\
 \Ext^1_X(\ZZ/n\ZZ,\mathbb{G}_{m,X})^\sim & \mbox{ if } i=2 \\
 \mu_n(K)^\sim & \mbox{ if } i=3 \\
 0 &  \mbox{ if } i>3 
 \end{cases}$$
where $^\sim$ denotes the Pontryagin dual of the corresponding group and $\mu_n(K)$ is the $n$th roots of unity in $K.$ For our purposes, we will need a more concrete description of $\Ext^1_X(\ZZ/n\ZZ, \mathbb{G}_{m,X}).$ Define the étale sheaf $$\DIV(X) = \oplus_p \ZZ_{/p} $$  on $X,$  where $p$ ranges over all closed points of $X$ and $\ZZ_{/p}$ denotes the skyscraper sheaf at that point. Note that $\Div X,$ the global sections of $\DIV X,$  is the ordinary free abelian group on the set of closed points of $X = \Spec \mathcal{O}_K.$ Let $j: \Spec K \rightarrow X$ be the canonical map induced from the inclusion $\mathcal{O}_K \rightarrow K.$  On $X_{et}$ (resp. $(\Spec K)_{et})$ we have the multiplicative group sheaf $\mathbb{G}_{m,X}$ (resp. $\mathbb{G}_{m,K}).$  Define the complex $\mathcal{C}^\bullet$ of étale sheaves on $X$ as $$ j_\ast \mathbb{G}_{m,K} \xrightarrow{\divis} \DIV X \rightarrow 0$$ (with $j_\ast \mathbb{G}_{m,K}$ in degree $0$ and the map $\divis$ as in \cite{MilneEtale}, II 3.9) and $\mathcal{E}^\bullet_n$ as the complex $$\ZZ \xrightarrow{\cdot n} \ZZ$$ of constant sheaves, with non-zero terms in degrees $-1$ and $0.$  Note that the obvious maps of complexes $$\mathbb{G}_{m,X} \rightarrow \mathcal{C}^\bullet$$ and $$\mathcal{E}^\bullet_n \rightarrow \ZZ/n\ZZ$$ are quasi-isomorphisms. Consider the complex $\HOM(\mathcal{E}^\bullet_n,\mathcal{C}^\bullet), $ which written out in components  is $$j_* \mathbb{G}_{m,K} \xrightarrow{(\cdot n,-\divis)} j_* \mathbb{G}_{m,K} \oplus \DIV X \xrightarrow{\divis+ \cdot n} \DIV X$$ where the first map multiplication by $n^{-1}$ on the first factor and $\divis$ on the second factor. The last map is the sum of multiplication by $\divis$ and multiplication by $n.$  We have that $$\HOM(\mathcal{E}^\bullet_n,\mathcal{C}^\bullet) \cong R \HOM(\ZZ/n\ZZ,\mathbb{G}_{m,X})$$ in $\Drv(\Sh(X_{et})),$ the derived category of $\Sh(X_{et})$ since $\mathcal{E}^\bullet_n$ is a locally free complex of abelian sheaves.
We will now use the hypercohomology spectral sequence for computing $$H^i(R \HOM (\ZZ/n\ZZ , \mathbb{G}_{m,X})) \cong \Ext^i_X(\ZZ/n\ZZ,\mathbb{G}_{m,X}).$$  Since $$R \Gamma(R \HOM(\ZZ/n\ZZ,\mathbb{G}_{m,X})) = R\Hom(\ZZ/n\ZZ,\mathbb{G}_{m,X})$$ the natural transformation $\Gamma \rightarrow R \Gamma$ induces the map $$\Gamma(X,\HOM(\mathcal{E}_n^\bullet,\mathcal{C}^\bullet)) \rightarrow R\Hom(\ZZ/n\ZZ,\mathbb{G}_{m,X})$$ in $\Drv(\Ab).$ Because $\Gamma(X,\Hom(\mathcal{E}_n^\bullet,\mathcal{C}^\bullet))$ is $2$-truncated, the map  $$\Gamma(X,\HOM(\mathcal{E}_n^\bullet,\mathcal{C}^\bullet)) \rightarrow R\Hom(\ZZ/n\ZZ,\mathbb{G}_{m,X})$$ factors through the $2$-truncation $$ \tau_{\leq 2} (\RHom_X(\ZZ/n\ZZ,\mathbb{G}_{m,X}))$$ of $\RHom_X(\ZZ/n\ZZ,\mathbb{G}_{m,X}).$
\begin{lemma} \label{quasiext}
Let $\HOM(\mathcal{E}^\bullet_n,\mathcal{C}^\bullet)$ be as above. Then the natural map  $$\Gamma(X,\HOM(\mathcal{E}_n^\bullet,\mathcal{C}^\bullet)) \rightarrow \tau_{\leq 2} (\RHom_X(\ZZ/n\ZZ,\mathbb{G}_{m,X}))$$ is an isomorphism in $\Drv(\Ab).$
\end{lemma}
\begin{proof}
We consider the hypercohomology spectral sequence $$E_2^{p,q} = H^p(H^q(\HOM(\mathcal{E}^\bullet_n,\mathcal{C}^\bullet))) \Rightarrow \Ext_X^{p+q}(\ZZ/n\ZZ,\mathbb{G}_m).$$  with $$E_1^{p,q} = H^q(X,\HOM(\mathcal{E}_n^\bullet,\mathcal{C}^\bullet)^p)$$  and show that the edge homomorphism  $E_2^{p,0} \rightarrow \Ext^p_X(\ZZ/n\ZZ,\mathbb{G}_{m,X})$ is an isomorphism for $p=0,1,2.$ The $E_1$-page of the spectral sequence can be visualized as follows$$\begin{sseq}{0...6}{0...6}
\ssdropbull
\ssarrow{1}{0}
\ssdropbull
\ssarrow{1}{0}
\ssdropbull
\ssmoveto{0}{2}
\ssdropbull
\ssarrow{1}{0}
\ssdropbull
\ssarrow{1}{0}
\ssdropbull
\end{sseq}$$ where the $\bullet$ means that the object at that corresponding position is non-zero. The differential $$E_1^{0,2} = H^2(X,j_* \mathbb{G}_{m,K}) \rightarrow  H^2(X,j_*\mathbb{G}_{m,K}) \oplus (\oplus_{p} \mathbb{Q}/\mathbb{Z}) = E_1^{1,2}$$  where $p$ ranges over all closed points, is injective (it is multiplication induced by $n^{-1}$ on the first factor, and the invariant map on the second factor). One then sees that $$E_1^{1,2} \rightarrow E_1^{2,2}$$ is surjective and that the $E_2$-page is as follows  $$\begin{sseq}{0...6}{0...6}
\ssdropbull
\ssmove 1 0
\ssdropbull
\ssmove 1 0
\ssdropbull
\ssmoveto{1}{2}
\ssdropbull
\end{sseq}$$ There can thus be no non-trivial differentials, so the spectral sequence collapses at $E_2$ and our lemma follows.
\end{proof}
\noindent For $a \in K^*$ let $\divis(a) \in \Div X$ be the element of $\Div X$ which we get by considering $a$ as a fractional ideal in $\mathcal{O}_K.$ 
\begin{corollary} \label{cor:extdescrip}
$$\Ext^i(\ZZ/n\ZZ,\mathbb{G}_{m,X}) = \begin{cases} \mu_n(K) & \mbox{ if } i= 0 \\ 
 Z_1/B_1 & \mbox{ if } i= 1 \\
 \Pic X /n & \mbox{ if } i=2 \\
 \ZZ/n\ZZ & \mbox{ if } i=3 \\
 0 &  \mbox{ if } i>3.
 \end{cases}$$
where $$Z_1 = \{(a,\mathfrak{a}) \in K^* \oplus \Div X| -\divis(a) = n\mathfrak{a}\}$$ and $$B_1 = \{(b^n,-\divis(b)) \in K^* \oplus \Div X | b \in K^* \}.$$  
\end{corollary}
\begin{proof}
The cases $i=0,1,2$ follows immediately from Lemma \ref{quasiext}, once we note that $H^1(\Gamma(X,\HOM(\mathcal{E}^\bullet_n,\mathcal{C}^\bullet)) \cong Z_1/B_1.$  The case when $i=3$ follows from that $H^0(X,\ZZ/n\ZZ) = \ZZ/n\ZZ$ so that $\Ext^3_X(\ZZ/n\ZZ,\mathbb{G}_{m,X}) \cong \ZZ/n\ZZ$ by Artin-Verdier duality.
\end{proof}
This corollary allows us to give a concrete description of $H^i(X,\ZZ/n\ZZ)$ for all $i.$ Namely,with $Z_1$ and $B_1$ as in Corollary \ref{cor:extdescrip}:
$$H^i(X,\ZZ/n\ZZ) = \begin{cases} \ZZ/n\ZZ & \mbox{ if } i= 0 \\ 
 (\Pic(X)/n)^\sim & \mbox{ if } i= 1 \\
 (Z_1/B_1)^\sim & \mbox{ if } i=2 \\
 \mu_n(K)^\sim & \mbox{ if } i=3 \\
 0 &  \mbox{ if } i>3.
 \end{cases}$$
\section{The cohomology ring of a totally imaginary number field }\label{section:cohomologyring}
\noindent In this section we will compute the cohomology ring $$\displaystyle H^*(X,\ZZ/2\ZZ)$$ for $X = \Spec \mathcal{O}_K$ the ring of integers of a totally imaginary number field $K$ under the isomorphisms $$H^i(X,\ZZ/2\ZZ) \cong \Ext_X^{3-i}(\ZZ/2\ZZ,\mathbb{G}_{m,X})^\sim$$ given by Artin-Verdier duality. To determine the cup product maps $$-\cup-: H^i(X,\ZZ/2\ZZ) \times H^j(X,\ZZ/2\ZZ) \rightarrow H^{i+j}(X,\ZZ/2\ZZ)$$ we see that since $H^n(X,\ZZ/2\ZZ) = 0$ for $n >3$ that we can restrict to when $i+j \leq 3.$ Since the cup product is graded commutative we can further assume without loss of generality that $i \geq j$ and since $H^0(X,\ZZ/2\ZZ) = \ZZ/2\ZZ$ is generated by the unit, we can also assume that $j \geq 1.$ To conclude, we only need to determine the cup product $$H^i(X,\ZZ/2\ZZ) \times H^j(X,\ZZ/2\ZZ) \rightarrow H^{i+j}(X,\ZZ/2\ZZ)$$ when $i=1,j=1$ and $i=2,j=1.$ The above discussion implies that to describe $\displaystyle H^*(X,\ZZ/2\ZZ)$ it is enough to determine the maps $$- \cup x:H^i(X,\ZZ/2\ZZ) \rightarrow H^{i+1}(X,\ZZ/2\ZZ)$$ for $i=1,2,$ and where $x \in H^1(X,\ZZ/2\ZZ)$ is arbitrary. From now on, we will denote the map $$- \cup x:H^i(X,\ZZ/2\ZZ) \rightarrow H^{i+1}(X,\ZZ/2\ZZ)$$ by $c_x.$ Let us fix an element $x \in H^1(X,\ZZ/2\ZZ)$ represented by the $\ZZ/2\ZZ$-torsor $p:Y \rightarrow X.$ Note that since $p$ is finite étale, the functors $p_*:\Sh(Y_{et}) \rightarrow \Sh(X_{et}),$ $p^* : \Sh(X_{et}) \rightarrow \Sh(Y_{et})$ which gives for any abelian étale sheaf $F$ on $X$ two maps $$N:p_*p^* F : \rightarrow F$$ and  $$u:F \rightarrow p_*p^*F,$$ both adjoint to the identity $p^* F \rightarrow p^*F.$ We will call $N$ the norm map.
\begin{lemma} \label{lemma:transfercup}
Let $X$ be a scheme and $$p:Y \rightarrow X$$ a $\ZZ/2\ZZ$-torsor corresponding to an element $x \in H^1(X,\ZZ/2\ZZ).$ Then the connecting homomorphism $$\delta_x:H^i(X,\ZZ/2\ZZ) \rightarrow H^{i+1}(X,\ZZ/2\ZZ)$$ arising from the short exact sequence of sheaves on $X$ \begin{equation} \label{xexact} 0 \rightarrow \ZZ/2\ZZ_X \xrightarrow{u} p_* p^* \ZZ/2\ZZ_X \xrightarrow{N} \ZZ/2\ZZ_X \rightarrow 0, \end{equation} is $c_x,$ the cup product with $x.$  
\end{lemma}
\begin{proof}
Note that $x \in H^1(X,\ZZ/2\ZZ),$ viewed as a $\ZZ/2\ZZ$-torsor, corresponds to a geometric morphism $$k_x:\Sh(X_{et}) \rightarrow B\ZZ/2\ZZ,$$ where $B\ZZ/2\ZZ$ is the topos of $\ZZ/2\ZZ$-sets. We have a universal $\ZZ/2\ZZ$-torsor on $B\ZZ/2\ZZ$ corresponding to $\ZZ/2\ZZ$ with $\ZZ/2\ZZ$ acting on itself by right translation. Call this universal torsor for $U_{\ZZ/2\ZZ}.$  On $B\ZZ/2\ZZ$ we have the exact sequence \begin{equation} 0 \rightarrow \ZZ/2\ZZ \rightarrow \ZZ/2\ZZ \oplus \ZZ/2\ZZ \rightarrow \ZZ/2\ZZ \rightarrow 0 \label{bexact} \end{equation} of $\ZZ/2\ZZ$-modules where $\ZZ/2\ZZ$ acts on $\ZZ/2\ZZ \oplus \ZZ/2\ZZ$ by taking $(1,0)$ to $(0,1).$ We claim that we can reduce the proposition to this universal case. To be more precise, our first claim is that the connecting homomorphism $$\delta: H^i(\ZZ/2\ZZ,\ZZ/2\ZZ)  \rightarrow H^{i+1}(\ZZ/2\ZZ,\ZZ/2\ZZ)$$ from exact sequence \ref{bexact} is cup product with the non-trivial element of $H^1(\ZZ/2\ZZ,\ZZ/2\ZZ),$  while our second claim is that exact sequence \ref{xexact} is the pull-back of exact sequence \ref{bexact} by $k_x.$  Our proposition follows  from these claims. Indeed, note that the cup product can be identified with Yoneda composition under the isomorphisms $H^i(X,\ZZ/2\ZZ) \cong \Ext^i_X(\ZZ/2\ZZ_X,\ZZ/2\ZZ_X) ,$  and that the connecting homomorphism 
$$H^i(X,\ZZ/2\ZZ) \rightarrow H^{i+1}(X,\ZZ/2\ZZ)$$ is given by Yoneda composition with the element in $\Ext^1_X(\ZZ/2\ZZ_X,\ZZ/2\ZZ_X)$ corresponding to exact sequence \ref{xexact}. The sufficiency of our claims now follows from the diagram $$\begin{tikzcd}  \arrow[d, shift right = 10 ex, "k_x^*"] \arrow[d, shift left= 10 ex, "k_x^*"]  \Ext^1_{\ZZ/2\ZZ}(\ZZ/2\ZZ,\ZZ/2\ZZ) \times \Ext^i_{\ZZ/2\ZZ}(\ZZ/2\ZZ,\ZZ/2\ZZ) \arrow[r] & \arrow[d,"k_x^*"]  \Ext^{i+1}_{\ZZ/2\ZZ}(\ZZ/2\ZZ,\ZZ/2\ZZ) \\ \Ext^1_X(\ZZ/2\ZZ_X,\ZZ/2\ZZ_X) \times \Ext^i_X(\ZZ/2\ZZ_X,\ZZ/2\ZZ_X) \arrow[r] & \Ext^{i+1}_X(\ZZ/2\ZZ_X,\ZZ/2\ZZ_X). \end{tikzcd}$$ To prove that the connecting homomorphism $\delta$ corresponds to cup product with the non-trivial element of $H^1(\ZZ/2\ZZ,\ZZ/2\ZZ), $ we first see that the exact sequence \ref{bexact} corresponds to an element $\beta  \in \Ext^1_{\ZZ/2\ZZ}(\ZZ/2\ZZ,\ZZ/2\ZZ),$ and that the connecting homomorphism  $$\Ext^i_{\ZZ/2\ZZ}(\ZZ/2\ZZ,\ZZ/2\ZZ) \rightarrow \Ext^{i+1}_{\ZZ/2\ZZ}(\ZZ/2\ZZ,\ZZ/2\ZZ)$$ given by exact sequence \ref{bexact} is the Yoneda product with $\beta.$ The Yoneda product coincides with the cup product, so after the identification $$H^i(\ZZ/2\ZZ,\ZZ/2\ZZ) \cong \Ext^i_{\ZZ/2\ZZ}(\ZZ/2\ZZ,\ZZ/2\ZZ)$$ we only need to see that $\beta \in \Ext^1_{\ZZ/2\ZZ}(\ZZ/2\ZZ,\ZZ/2\ZZ)$ corresponds to the non-trivial element, i.e that sequence \ref{bexact} is non-split, which is immediate. To prove that exact sequence \ref{xexact} is the pullback of exact sequence \ref{bexact}, the only non-trivial claim is that $$k_x^\ast(\ZZ/2\ZZ \oplus \ZZ/2\ZZ) = p_*p^*\ZZ/2 \ZZ_X.$$ For proving this, let $$B\ZZ/2\ZZ_{/U_{\ZZ/2\ZZ}}$$ be the topos of $\ZZ/2\ZZ$-set over $U_{\ZZ/2\ZZ}.$  We have a geometric morphism $$u=(u^*,u_*):B\ZZ/2\ZZ_{/U_{\ZZ/2\ZZ}} \rightarrow B\ZZ/2\ZZ$$  where $$u^*(A) = A \times U_{/\ZZ/2\ZZ}.$$ One can then verify that $$u_*u^* (\ZZ/2\ZZ) = \ZZ/2\ZZ \oplus \ZZ/2\ZZ$$ with our previously defined action.  Note now that $$k_x : Sh(X_{et}) \rightarrow B\ZZ/2\ZZ$$ induces a map $$\tilde{k}_x:Sh(Y_{et}) \rightarrow  B\ZZ/2\ZZ_{/U_{\ZZ/2\ZZ}} $$ and that we have a pullback diagram of topoi $$\begin{tikzcd} Sh(Y_{et}) \arrow[r,"\tilde{k}_x"] \arrow[d,"p"] & B\ZZ/2\ZZ_{/U_{\ZZ/2\ZZ}} \arrow[d,"u"] \\ Sh(X_{et}) \arrow[r,"k_x"] & B \ZZ/2\ZZ .\end{tikzcd}$$  We now claim that there is a natural isomorphism of functors $$p_* \tilde{k}_x^* \cong k_x^* u_*,$$ this is what is known as a Beck-Chevalley condition. This can either be proven directly or follows from that $u$ is a ''tidy'' geometric morphism, as developed in \cite{MoerjdikProper} (see also \cite[C3.4]{JohnstoneElephant2} for a textbook account). Given this, we now see that $$k_x^\ast(u_*u^* \ZZ/2\ZZ) \cong p_* \tilde{k}_x^* u^* \ZZ/2\ZZ \cong p_*p^* k_x^* \ZZ/2\ZZ $$ and thus that $$k_x^\ast(u_u^*\ZZ/2\ZZ) \cong p_*p^* \ZZ/2\ZZ_X,$$ so exact sequence \ref{xexact} is the pullback of the exact sequence \ref{bexact}, and our proposition follows. 
\end{proof}
By Lemma \ref{dualcup} applied to the map $\delta_x: \ZZ/2\ZZ_X \rightarrow \ZZ/2\ZZ_X[1],$ we see that to compute $$c_x:H^i(X,\ZZ/2\ZZ) \rightarrow H^{i+1}(X,\ZZ/2\ZZ)$$ it is enough to calculate the map $$c_x^\sim : \Ext_X^{3-(i+1)}(\ZZ/2\ZZ_X,\mathbb{G}_{m,X}) \rightarrow \Ext_X^{3-i}(\ZZ/2\ZZ_X,\mathbb{G}_{m,X})$$ which is under Artin-Verdier duality Pontryagin dual to $c_x.$ We now note that $c_x^\sim$ is induced by applying $H^{3-i}$ to the map $$\RHom_X(\delta_x,\mathbb{G}_{m,X}): \RHom_X(\ZZ/2\ZZ_X[1],\mathbb{G}_{m,X}) \rightarrow \RHom_X(\ZZ/2\ZZ_X,\mathbb{G}_{m,X}).$$ Our plan is now to compute $\RHom_X(\delta_x,\mathbb{G}_{m,X})$ by first resolving $\ZZ/2\ZZ_X$ and then take a resolution of $\mathbb{G}_{m,X}.$ In what follows, if we have a morphism $f$ of complexes of étale sheaves, we will let $\Cone(f)$ be the mapping cone of $f.$ Consider now the map 
$$\delta_x: \ZZ/2\ZZ_X \rightarrow \ZZ/2\ZZ_X[1]$$ and lift $\delta_x$ to the zig-zag \begin{equation} \label{ziggystardust} \ZZ/2\ZZ_X \xleftarrow{q(u)} \Cone(u) \xrightarrow{\pi(u)} \ZZ/2\ZZ_X[1] \end{equation}  in the category of complexes of étale sheaves, where $\pi(u)$ is the canonical projection, and $q(u)$ is projection onto $p_*p^*\ZZ/2\ZZ_X$ followed by the norm map $$p_*p^* \ZZ/2\ZZ_X \rightarrow \ZZ/2\ZZ_X.$$  It is easy to see that $q(u)$ is a quasi-isomorphism and that this zig-zag lifts $\delta_x$ in the sense that if we denote by $$\gamma:\Ch(X_{et}) \rightarrow \Drv(X_{et})$$ the localization functor from the category of complexes of étale sheaves, then $\gamma(\pi(u)) \circ \gamma(q(u))^{-1}$ in $\Drv(X_{et})$ is $\delta_x.$ Consider now the resolutions $$\mathcal{E}_2 = (\ZZ_X \xrightarrow{\cdot 2} \ZZ_X)$$ and $p_*p^* \mathcal{E}_2$ of $\ZZ/2\ZZ_X$ and $p_*p^* \ZZ/2\ZZ_X$ respectively. We have just as for $\ZZ/2\ZZ_X$ two canonical maps $$\hat{u}: \mathcal{E}_2 \rightarrow p_* p^* \mathcal{E}_2, \hat{N}:p_*p^* \mathcal{E}_2 \rightarrow \mathcal{E}_2.$$ In preparation of a lemma, let us record some elementary facts. First,  if $Z$ is a complex of etale sheaves, a map $f:p_*p^* \mathcal{E}_2 \rightarrow Z$ together with a null-homotopy $k$ of $f\hat{u}:\mathcal{E}_2 \rightarrow Z$ defines a map $(f,k):\Cone(\hat{u}) \rightarrow Z.$ Second, the composite $$\hat{N} \hat{u}:\mathcal{E}_2 \rightarrow \mathcal{E}_2$$ is multiplication by $2.$  With these remarks, the following is a straightforward calculation.
\begin{lemma}\label{zigrep}
Let $\Cone(\hat{u})$ be the mapping cone of $\hat{u}:\mathcal{E}_2 \rightarrow p_*p^* \mathcal{E}_2$ and choose a nullhomotopy $h$ of $\mathcal{E}_2 \xrightarrow{\cdot 2} \mathcal{E}_2.$ Then the map $$q(\hat{u}) = (\hat{N},h):\Cone(u) \rightarrow \mathcal{E}_2$$ is a quasi-isomorphism.
\end{lemma}
By Lemma \ref{zigrep}, if we let $\pi(\hat{u}):\Cone(\hat{u}) \rightarrow \mathcal{E}_2[1]$ be the projection, we can represent $\delta_X : \ZZ/2\ZZ_X \rightarrow \ZZ/2\ZZ_X[1]$ by the zig-zag \begin{equation} \label{ziggy1} \mathcal{E}_2 \xleftarrow{q(\hat{u})} \Cone(\hat{u}) \xrightarrow{\pi(\hat{u})} \mathcal{E}_2[1] \end{equation} in the category of complexes of étale sheaves on $X.$ Recall now the resolution $\mathcal{C}^\bullet$ of $\mathbb{G}_{m,X}$ from Section \ref{section:lemmasonetale} and apply $\HOM(-,\mathcal{C}^\bullet)$ to \ref{ziggy1} to get $$ \HOM(\mathcal{E}_2,\mathcal{C}^\bullet) \xrightarrow{\HOM(q(\hat{u}),\mathcal{C}^\bullet)} \HOM(\Cone(\hat{u}),\mathcal{C}^\bullet) \xleftarrow{\HOM(\pi(u),\mathcal{C}^\bullet)} \HOM(\mathcal{E}_2[1],\mathcal{C}^\bullet). $$ Note now that $\HOM(q(\hat{u}),\mathcal{C}^\bullet)$ is a quasi-isomorphism since $$q(\hat{u}):\Cone(\hat{u}) \rightarrow \mathcal{E}_2$$ is a quasi-isomorphism between complexes of locally free sheaves.  Apply the global sections functor to get $$ \Hom(\mathcal{E}_2,\mathcal{C}^\bullet) \xrightarrow{\Hom(q(\hat{u}),\mathcal{C}^\bullet)} \Hom(\Cone(\hat{u}),\mathcal{C}^\bullet) \xleftarrow{\Hom(\pi(u),\mathcal{C}^\bullet)} \Hom(\mathcal{E}_2[1],\mathcal{C}^\bullet).$$ Putting together the facts that $R \Gamma \circ R \HOM = \RHom_X$ and $R \HOM(\mathcal{E}_2,-) = \HOM(\mathcal{E}_2,-)$  with the natural transformation $\Gamma \rightarrow R\Gamma$ we get the commutative diagram 
$$\begin{tikzcd}[column sep = small] \Hom(\mathcal{E}_2,\mathcal{C}^\bullet)  \arrow[d] \arrow[r,"{\Hom(q(\hat{u}),\mathcal{C}^\bullet)}"] & \Hom(\Cone(\hat{u}),\mathcal{C}^\bullet) \arrow[d,"s"] & \arrow[l,"{\Hom(\pi(u),\mathcal{C}^\bullet)}",swap] \arrow[d] \Hom(\mathcal{E}_2[1],\mathcal{C}^\bullet) \\ 
\RHom_X(\ZZ/2\ZZ_X,\mathbb{G}_{m,X}) \arrow[r] & \RHom_X(\Cone(u),\mathbb{G}_{m,X}) & \arrow[l] \RHom_X(\ZZ/2\ZZ_X[1] , \mathbb{G}_{m,X}). \end{tikzcd}$$ 
The lower horizontal maps in this diagram comes from applying $\RHom_X(,\mathbb{G}_{m,X})$ to the zig-zag \ref{ziggystardust} which represented $\delta_x:\ZZ/2\ZZ_X \rightarrow \ZZ/2\ZZ_X[1].$ We now make the remark, which will be used in the proof of the following lemma as well as in the description of $$c_x:H^i(X,\ZZ/2\ZZ) \rightarrow H^{i+1}(X,\ZZ/2\ZZ),$$ that $x \in H^1(X,\ZZ/2\ZZ)$ can be represented by a $\ZZ/2\ZZ$-torsor $$p:Y = \Spec \mathcal{O}_L \rightarrow \Spec \mathcal{O}_K$$ where $Y = \Spec \mathcal{O}_L$ is the ring of integers of a quadratic unramified extension.  We will also say that $Y$ represents the $\ZZ/2\ZZ$-torsor, leaving the map $p$ implicit. We have just as for $X = \Spec \mathcal{O}_K$ the sheaves $\DIV Y$ and $j'_* \mathbb{G}_{m,L}$ on $Y,$ where $j': \Spec L \rightarrow \Spec \mathcal{O}_K$ is the canonical map.
\begin{lemma}
The map $s$ and hence $\Hom(q(\hat{u}),\mathcal{C}^\bullet),$ induces an isomorphism on $H^i$ for $i=0,1,2.$
\end{lemma}
\begin{proof}
We know that the map $\Hom(\mathcal{E}_2,\mathcal{C}^\bullet) \rightarrow \RHom_X(\ZZ/2\ZZ_X,\mathbb{G}_{m,X})$ induces an isomorphism on $H^i$ for $i=0,1,2$ by Lemma \ref{quasiext}. For the convenience of the reader, let us note that $\HOM(\Cone(\hat{u}),\mathcal{C}^\bullet)$ is isomorphic to $\Cone(\HOM(\mathcal{E}_2,N))[-1]$ where $N:p_*p^* \mathbb{G}_{m,X} \rightarrow \mathbb{G}_{m,X}$ is the norm map. One now runs the hypercohomology spectral sequence on $\HOM(\Cone(\hat{u}),\mathcal{C}^\bullet)$ and sees that the $E_1$-page can be visualized as follows 
$$\begin{sseq}{0...6}{0...6}
\ssdropbull
\ssarrow{1}{0}
\ssdropbull
\ssarrow{1}{0}
\ssdropbull
\ssarrow{1}{0}
\ssdropbull
\ssmoveto{0}{2}
\ssdropbull
\ssarrow{1}{0}
\ssdropbull
\ssarrow{1}{0}
\ssdropbull
\ssarrow{1}{0}
\ssdropbull
\end{sseq}$$ where the $\bullet$ means that the object at that corresponding position is non-zero. As in the proof of Lemma \ref{quasiext}, the differential $$d_1:E_1^{0,2} \rightarrow E_1^{1,2}$$ is injective. The differential $d_1:E_1^{0,2} \rightarrow E_1^{1,2}$ can then be identified with the map $$H^2(Y,j'_* \mathbb{G}_{m,L}) \rightarrow  H^2(X,j_*\mathbb{G}_{m,L}) \oplus H^2(Y,\DIV Y) \oplus H^2(Y,j'_*\mathbb{G}_{m,L})$$ which is the map induced by the inverse of the norm on the first component, the invariant map on the second component and the map induced by multiplication by $2^{-1}$ on the third component.  This map is injective since the invariant map 
$$H^2(Y,j'_*\mathbb{G}_{m,L}) \rightarrow H^2(Y,\DIV Y)$$ is injective. On the $E_2$-page we see thus that no differential can hit $E_2^{p,0}$ for $p=0,1,2.$ This implies that the edge homomorphism is an isomorphism in degrees $p=0,1,2$ so that $H^i(s)$ induces an isomorphism in the stated degrees. The fact that $H^i(\pi(u))$ is an isomorphism as well follows directly from that both $H^i(\pi(u))$ and the composite $H^i(s) \circ H^i(\pi(u))$ are isomorphisms. 
\end{proof}
\begin{corollary}
The map $\Ext^i(\ZZ/2\ZZ_X,\mathbb{G}_{m,X}) \rightarrow \Ext^{i+1}(\ZZ/2\ZZ_X,\mathbb{G}_{m,X})$ for $i=0,1,$ is isomorphic to the map $H^i(\pi(u))^{-1} \circ H^i(t).$ 
\end{corollary}
Note that if $Y = \Spec \mathcal{O}_L$ represents $x \in H^1(X,\ZZ/2\ZZ)$ we can write $L = K(\sqrt{c})$ for some $c \in K^*$ with the divisor of $c$ even, i.e $\divis(c) = 2 \mathfrak{c}$ for some $\mathfrak{c} \in \Div(X).$ We thus get for every $x \in H^1(X,\ZZ/2\ZZ)$ an element $c \in K^*,$ well-defined up to squares. We provide a proof of the more involved of the following two corollaries, leaving the first to the reader.
\begin{corollary}\label{cor:cap1}
Let $x \in H^1(X,\ZZ/2\ZZ)$ and identify $x$ with an unramified quadratic extension $L = K(\sqrt{a})$ with $a \in K^*$ such that $\divis(a) = 2\mathfrak{a}$ for some $\mathfrak{a} \in \Div(X).$ The map $$c_x^\sim: \Ext_X^0(\ZZ/2\ZZ,\mathbb{G}_{m,X}) \rightarrow \Ext_X^1(\ZZ/2\ZZ,\mathbb{G}_{m,X})$$ then takes $$-1 \in \mu_2(K) \cong \Ext^0_X(\ZZ/2\ZZ,\mathbb{G}_{m,X})$$ to $$(a^{-1},\mathfrak{a}) \in H^1(\Hom(\mathcal{E}^\bullet_2,\mathcal{C}^\bullet))  \cong \Ext^1_X(\ZZ/2\ZZ,\mathbb{G}_{m,X}).$$  
\end{corollary}
\begin{proof}
We have the zig-zag $$\begin{tikzcd}[column sep = huge] \Hom(\mathcal{E}_2,\mathcal{C}^\bullet)  \arrow[r,"{\Hom(q(\hat{u}),\mathcal{C}^\bullet)}"] & \Hom(\Cone(\hat{u}),\mathcal{C}^\bullet) & \arrow[l,"{\Hom(\pi(u),\mathcal{C}^\bullet)}",swap]  \Hom(\mathcal{E}_2[1],\mathcal{C}^\bullet), \end{tikzcd}$$ and we note that the map $$g= \Hom(q(\hat{u}),\mathcal{C}^\bullet): \Hom(\mathcal{E}_2,\mathcal{C}^\bullet) \rightarrow \Hom(\Cone(\hat{u}),\mathcal{C}^\bullet)$$ is isomorphic to the map given in components as follows
$$ \begin{tikzcd} K^* \arrow[r,"g^0"] \arrow[d,"d^0"] & L^* \arrow[d,"{d^0_{\Hom(\Cone(\hat{u}),\mathcal{C}^\bullet)}}"] \\ K^* \oplus \Div X \arrow[r,"g^1"] \arrow[d,"{d^1}"] & K^* \oplus (\Div Y \oplus L^*) \arrow[d,,"{d^1_{\Hom(\Cone(\hat{u}),\mathcal{C}^\bullet)}}"] \\ \Div X \arrow[r,"g^2"] \arrow[d] & (K^* \oplus \Div X) \oplus \Div Y \arrow[d,,"{d^2_{\Hom(\Cone(\hat{u}),\mathcal{C}^\bullet)}}"] \\ 0 \arrow[r] & \Div X \end{tikzcd} $$ where $$g^0(a) = i(a),$$ $$g^1(a,\mathfrak{a}) = (a,i(\mathfrak{a}),i(a))$$ and $$g^2(\mathfrak{a}) = (1,\mathfrak{a},i(\mathfrak{a}))$$ and the unadorned differentials are the differentials coming from $$\Hom(\mathcal{E}^\bullet_2,\mathcal{C}^\bullet).$$ The differentials for $\Hom(\Cone(\hat{u}),\mathcal{C}^\bullet)$ are given by $$d^0_{\Hom(\Cone(\hat{u}),\mathcal{C}^\bullet)}(a) = (N_{L/K}(a)^{-1},\divis(a),a^{-2})$$ while  $d^1_{\Hom(\Cone(\hat{u}),\mathcal{C}^\bullet)}(a,\mathfrak{b},c)$ is equal to  $$(a^{-2} \cdot N_{L/K}(c),\divis(a)+N_{L/K}(\mathfrak{b}),2\mathfrak{b}+\divis(c))$$ and
$$d^2_{\Hom(\Cone(\hat{u}),\mathcal{C}^\bullet)}(a,\mathfrak{a},\mathfrak{b}) = \divis(a)-2\mathfrak{a}+N_{L/K}(\mathfrak{b}).$$  To see that the complex to the right indeed is  $\Hom(\Cone(\hat{u}),\mathcal{C}^\bullet)$  one simply uses the isomorphisms $$\Hom_X(\ZZ_X,j_* \mathbb{G}_{m,K}) \cong K^*, \Hom_X(p_*p^* \ZZ_X,j_* \mathbb{G}_{m,K}) \cong L^*$$ and $$\Hom_X(\ZZ_X,\DIV X) \cong \Div X, \HOM_X(p_*p^* \ZZ_X , \DIV X ) \cong \Div Y$$ and that the differentials are as claimed follows easily. We now see that $\Hom(\pi(u),\mathcal{C}^\bullet)(-1)$ is the cycle $(-1,1,1).$ If we reduce $(-1,1,1)$ by $d^0(\sqrt{a}) = (-a^{-1},\divis(\sqrt{a}),a^{-1})$ we get the cycle $$(a^{-1},\divis(\sqrt{a}),a^{-1}) = (a,\divis(i(\mathfrak{a})),a) = g^1(a^{-1},\mathfrak{a}).$$ This shows that $(-1,1,1)$ goes to $(a^{-1},\mathfrak{a})$ so our Corollary follows.
\end{proof}
\begin{corollary} \label{cor:cap2}
Let $x \in H^1(X,\ZZ/2\ZZ)$ and identify $x$ with an unramified quadratic extension $Y =  \Spec \mathcal{O}_L$ where $L=K(\sqrt{a}).$ The map $$c_x^\sim: \Ext^1_X (\ZZ/2\ZZ,\mathbb{G}_{m,X}) \rightarrow \Ext^2_X(\ZZ/2\ZZ,\mathbb{G}_{m,X})$$ then takes $$(b,\mathfrak{b}) \in  H^1(\Hom(\mathcal{E}^\bullet_2,\mathcal{C}^\bullet)) \cong \Ext^1_X(\ZZ/2\ZZ,\mathbb{G}_{m,X})$$ to $$\mathfrak{b}+N(\mathfrak{b}') \in \Pic X/2 \Pic X \cong \Ext^2_X(\ZZ/2\ZZ,\mathbb{G}_{m,X})$$  where $\mathfrak{b}' \in \Pic(Y) $ is an element such that $$\mathfrak{b} = \mathfrak{b}'-\mathfrak{b}'^\sigma$$ in $\Pic(Y)$ for $\sigma \in \Gal(L/K)$ a generator.
\end{corollary}
\begin{proof}
 We use notation as in the proof of Corollary \ref{cor:cap1}. Let $z=(b, \mathfrak{b})$ be a cycle representing the class $(b,\mathfrak{b}) \in H^1(\Hom(\mathcal{E}^\bullet_2,\mathcal{C}^\bullet)).$  Then $\Hom(\pi(u),\mathcal{C}^\bullet)(z)$ is the cycle $(b,\mathfrak{b},1) \in (K^* \oplus \Div X) \oplus \Div Y.$ We want to find the cycle which is, modulo boundaries, congruent to a cycle of the form $(1,\mathfrak{a},i(\mathfrak{a}))$ where $\mathfrak{a} \in \Div X.$ Note that since $$\divis(b)+2\mathfrak{b}=0$$ we see that $$i(\mathfrak{b}) \in \ker N_{L/K} : \Pic Y \rightarrow \Pic X.$$ By Furtwängler's theorem  \cite[IV, Thm. 1]{LemmermeyerPrincipal} gives thus that $$i(\mathfrak{b}) = (\mathfrak{b'}-\mathfrak{b'}^\sigma)+ \divis(t)$$ for $\mathfrak{b'}$ an ideal in $L,t \in L^*$ and $\sigma \in \Gal(L/K)$ a generator.  By applying the norm to both sides we see that $$\divis(N(t)) = -\divis(b)$$ so that $N(t) = b^{-1} u$ for $u$ a unit of $K.$ Since $L/K$ is an unramified extension, by Hasse's norm theorem \cite{HasseNorm} and the fact that units are always norms in unramified extensions of local fields, there is a $v \in L^*$ such that $N(v) = u^{-1}.$ If we now mod out $(b,\mathfrak{b},1)$ by the image of $(1,\mathfrak{b}',v \cdot  t)$ under $d^1_{\Cone(\Hom(\hat{u}),\mathcal{C}^\bullet)}$ we get the new tuple $$(1,\mathfrak{b}+N(\mathfrak{b}'),2\mathfrak{b}' +\divis(t))$$ and this is in the image of $\Hom(\pi(u),\mathcal{C}^\bullet),$ hence we get our Corollary.

\end{proof}
By \ref{dualcup} this gives us a description of the cup product as well. In the following propositions, given $y \in \Ext^{3-i}_X(\ZZ/2\ZZ_X,\mathbb{G}_{m,X})^\sim$ and $z \in \Ext^{3-i}_X(\ZZ/2\ZZ_X,\mathbb{G}_{m,X})$ we will let $ \langle y, z \rangle $ denote the evaluation of $y$ on $z.$
\begin{prop} \label{prop:cupgen}
Let $X = \Spec \mathcal{O}_K$ be a totally imaginary number field and identify $H^2(X,\ZZ/2\ZZ)$ with $$\Ext^1_X(\ZZ/2\ZZ_X,\mathbb{G}_{m,X})^\sim$$ where $^\sim$ denotes the Pontryagin dual and let $x \in H^1(X,\ZZ/2\ZZ)$ be a cohomology class represented by the unramified quadratic extension $Y = \Spec \mathcal{O}_L$ where $L= K(\sqrt{c}).$  Then, for $$y \in H^2(X,\ZZ/2\ZZ),$$  we have that $$c_x(y) = x \cup y \neq 0$$ if and only if $\langle y , (c,\mathfrak{c}^{-1}) \rangle \neq 0,$ where $\mathfrak{c}$  is an ideal of $\Spec \mathcal{O}_K$  such that $\divis(c) = 2\mathfrak{c}.$
\end{prop}
\begin{proof}
We calculate $$c_x : \Ext^1_X(\ZZ/2\ZZ_X,\mathbb{G}_{m,X})^\sim \rightarrow \Ext^0_X(\ZZ/2\ZZ,\mathbb{G}_{m,X})^\sim.$$ For $y \in H^2(X,\ZZ/2\ZZ) \cong \Ext^1_X(\ZZ/2\ZZ_X,\mathbb{G}_{m,X})^\sim$ Corollary ~\ref{cor:cap1} gives that$$c_x(y) \in \Ext^0_X(\ZZ/2\ZZ_X,\mathbb{G}_{m,X})^\sim$$ is non-zero if and only if $$\langle y,  c_x^\sim(-1)  \rangle =  \langle y, (c,\mathfrak{c}^{-1}) \rangle \neq 0.$$ 
\end{proof}
\begin{prop} \label{prop:doublecup}
Let $X = \Spec \mathcal{O}_K$ be a totally imaginary number field and identify $H^2(X,\ZZ/2\ZZ)$ with $$ \Ext^1_X(\ZZ/2\ZZ_X,\mathbb{G}_{m,X})^\sim$$ where $^\sim$ denotes the Pontryagin dual and let $x \in H^1(X,\ZZ/2\ZZ)$ be a cohomology class represented by the unramified quadratic extension $Y = \Spec \mathcal{O}_L$ where $L=K(\sqrt{c}).$ For $y$ in $ H^1(X,\ZZ/2\ZZ) \cong (\Pic X / 2 \Pic X)^\sim,$ represented by the unramified extension $Y' = \Spec \mathcal{O}_{Q},$ $$c_x(y) \in \Ext^1_X(\ZZ/2\ZZ_X,\mathbb{G}_{m,X})^\sim$$ is given by the formula $$\langle c_x(y), (a, \mathfrak{a}) \rangle =\langle y, \mathfrak{a}+ N_{L/K}(\mathfrak{a'}) \rangle $$  where $(a,\mathfrak{a}) \in \Ext^1_X(\ZZ/2\ZZ_X,\mathbb{G}_{m,X})$ and $$\mathfrak{a}' \in \Pic(Y)$$ is such that $ \mathfrak{a} = \mathfrak{a}'/\mathfrak{a}'^\sigma$ for $\sigma \in \Gal(L/K)$ a generator. In particular, $$<c_x(y),(a,\mathfrak{a})> =0$$ if and only if $\mathfrak{a}+N_{L/K}(\mathfrak{a}')$ is in the image of $N_{Q/K}.$
\end{prop}
\begin{proof}
We calculate $$c_x:\Ext^2_X(\ZZ/2\ZZ_X,\mathbb{G}_{m,X})^\sim \rightarrow \Ext^1_X(\ZZ/2\ZZ_X,\mathbb{G}_{m,X})^\sim.$$ By Artin reciprocity $y \in (\Pic X/2 \Pic X)^\sim$ corresponds to an map $\Pic(X)/2\Pic(X) \rightarrow \ZZ/2\ZZ$ with kernel $N_{Q/K} (\Pic Y').$ Now, Corollary \ref{cor:cap2} gives us that $$x \cup y = c_x(y) \in \Ext^1_X(\ZZ/2\ZZ_X,\mathbb{G}_{m,X})^\sim$$ is the element given by the formula $$\langle c_x,(a,\mathfrak{a}) \rangle = \langle y, c_x^\sim(a,\mathfrak{a})) \rangle =  \langle y,  a +N(\mathfrak{a}') \rangle$$ where $$(a, \mathfrak{a}) \in \Ext^1_X(\ZZ/2\ZZ_X,\mathbb{G}_{m,X})$$ and $\mathfrak{a}'$ is as in the proposition. This proves our claimed formula and the fact that $$<c_x(y),(a,\mathfrak{a})> = 0$$ if and only if $\mathfrak{a}+N_{L/K}(\mathfrak{a'})$ is in the image of $$N_{Q/K}:\Pic Y' \rightarrow \Pic X.$$ 
\end{proof}
\begin{remark} \label{rmk:samecup}
Note that in Proposition \ref{prop:doublecup}, if we have that $L=Q,$ then to check whether $<x \cup x,(a,\mathfrak{a})> =0,$ it is enough to see that $\mathfrak{a}$ is inert in $L,$ since then $\mathfrak{a}$ is not in the image of the norm map by Artin reciprocity.
\end{remark} 
\begin{corollary}
Let $X = \Spec \mathcal{O}_K$ be a totally imaginary number field and identify $H^2(X,\ZZ/2\ZZ)$ with $ \Ext^1_X(\ZZ/2\ZZ_X,\mathbb{G}_{m,X})^\sim$ where $^\sim$ denotes the Pontryagin dual and let $x \in H^1(X,\ZZ/2\ZZ)$ be a cohomology class represented by the unramified quadratic extension $Y = \Spec \mathcal{O}_L$ where $L=K(\sqrt{c}).$ Then $$x \cup x \neq 0$$ if and only if some $2$-torsion element in $\Pic(X)$ is not in the image of $N_{L|K}:\Pic(Y) \rightarrow \Pic(X).$
\end{corollary}
We can now prove Proposition \ref{prop:triplecup} which we recall for the convenience of the reader. \triplecup*
\begin{proof}[Proof of Proposition \ref{prop:triplecup}]
By abuse of notation, we will also denote by $x$ the canonical element of $(\Pic(X)/2\Pic(X))^\sim$ associated to $p:Y \rightarrow X.$  Proposition ~\ref{prop:cupgen} and Remark \ref{rmk:samecup} gives us that $x\cup x \cup y$ is non-zero if and only if $\langle c_x(x), (d,\divis(d)/2) \rangle \neq 0$ which holds by Proposition \ref{prop:doublecup} if and only if $ \langle x, \divis(d)/2\rangle \neq 0$ which is true precisely when $\divis(d)/2$ is not in the image of the norm map $$N_{L|K}: \Pic Y \cong Cl(L) \rightarrow \Pic X \cong Cl(K).$$ But note that if we write $\divis(d)/2  = \prod_i \mathfrak{p}_i^{e_i},$ it is easy to see that $\divis(d)/2$ is in the image of the norm map if and only if $\prod_{\mathfrak{p}_i \text{ inert in } L} \mathfrak{p_i}^{e_i}$ is. We know that the Artin symbol $$\genfrac(){}{0}{L/K}{}:Cl(K) \rightarrow \ZZ/2\ZZ$$ vanishes on an ideal class $I$ when $I$  is in the image of $$N_{L|K} : Cl(L) \rightarrow Cl(K).$$ Artin reciprocity then gives that for a prime ideal $\mathfrak{p}$ of $\mathcal{O}_K,$  $$\genfrac(){}{0}{L/K}{\mathfrak{p}} = 0$$  exactly when $\mathfrak{p}$ is inert in $L.$ This yields our proposition since the Artin symbol is additive. 
\end{proof}

\nocite{EmbeddingFaddeev}
\nocite{SchlankBrauer}
\nocite{MilneADT}
\bibliographystyle{plain}
\bibliography{cupproducts}{}

\end{document}